\documentclass[pdflatex,sn-basic]{sn-jnl}% Math and Physical Sciences Numbered Reference Style
\usepackage{graphicx}%
\usepackage{multirow}%
\usepackage{amsmath,amssymb,amsfonts}%
\usepackage{amsthm}%
\usepackage{mathrsfs}%
\usepackage[title]{appendix}%
\usepackage{xcolor}%
\usepackage{textcomp}%
\usepackage{manyfoot}%
\usepackage{booktabs}%
\usepackage{algorithm}%
\usepackage{algorithmicx}%
\usepackage[noend]{algpseudocode}%
\usepackage{listings}%
\usepackage{soul}
\usepackage{preamble}

\theoremstyle{thmstyleone}%
\newtheorem{theorem}{Theorem}%  meant for continuous numbers
\newtheorem{proposition}[theorem]{Proposition}% 
\newtheorem{lemma}[theorem]{Lemma}
\newtheorem{definition}[theorem]{Definition}%

\theoremstyle{thmstyletwo}%
\newtheorem{example}{Example}%
\newtheorem{claim}{Claim}
 
\theoremstyle{thmstylethree}%

\begin{document}

\title[Tropical Fermat-Weber Points]{Tropical Fermat-Weber Points over \edit{Spaces of $M$-Ultrametrics}}


\author[1]{\fnm{Shelby} \sur{Cox}}\email{spcox@umich.edu}
\equalcont{These authors contributed equally to this work.}

\author[2]{\fnm{John} \sur{Sabol}}\email{john.sabol@nps.edu}
\equalcont{These authors contributed equally to this work.}

\author[3]{\fnm{Roan} \sur{Talbut}}\email{r.talbut21@imperial.ac.uk}
\equalcont{These authors contributed equally to this work.}

\author*[2]{\fnm{Ruriko} \sur{Yoshida}}\email{ryoshida@nps.edu}

\affil[1]{\orgname{The Max Planck Institute for Mathematics in the Sciences}, \orgaddress{\street{Inselstraße 22}, \city{Leipzig}, \postcode{04103},  \country{Germany}}}

\affil[2]{\orgdiv{Operations Research Department}, \orgname{Naval Postgraduate School}, \orgaddress{\street{1411 Cunningham Road}, \city{Monterey}, \postcode{93943}, \state{CA}, \country{USA}}}

\affil[3]{\orgdiv{Mathematics Department}, \orgname{Imperial College London}, \orgaddress{\street{South Kensington Campus}, \city{London}, \postcode{SW7 2AZ}, \country{UK}}}


\abstract{\edit{We extend reconstruction methods for phylogenetic trees to ultrametrics of arbitrary matroids and study the stability of these data analysis methods in the combinatorial spirit of Andreas Dress.
In particular, we generalize Atteson's work on the safety radius of phylogenetic reconstruction methods, as well as Gascuel and Steel's work on the stochastic safety radius, to arbitrary matroids.
We also show that although the tropical Fermat-Weber points of an $M$-ultrametric sample are generally not contained in the space of $M$-ultrametrics, the intersection between the Fermat-Weber set and the space of $M$-ultrametrics is non-empty.}}




\maketitle
\begin{center}
{\large \textit{Dedicated to Andreas Dress}}    
\end{center}

\section{Introduction}

Andreas Dress was a pioneer at the intersection of mathematics and evolutionary biology. \edit{In particular, he} made significant contributions to the combinatorics of phylogenetic trees (for example, \citep{dress2014associated}). 
A phylogenetic tree is a metric tree with non-negatively weighted \rot{edges, which is often used to represent} evolutionary relationships between taxa. 
It is well known that the space of phylogenetic trees with $p$ leaves is the Dressian $\edit{\mathrm{Dr}}(2, p)$. 
\rot{In \citep{dress2005delta}, Andreas Dress studied {\em ultrametrics}, or equivalently, {\em equidistant trees}, which are rooted phylogenetic trees whose total branch length from the root to each leaf is the same for all $p$ leaves}.
In 2006, Ardila and Klivans showed that the space of \edit{ultrametrics} is the Bergman fan associated with the matroid \edit{of} the complete graph on $p$ vertices \citep{arilda2006bergmancomplex}.

In phylogenetics, one of the main tasks is to reconstruct a phylogenetic tree from an alignment.  
A {\em distance-based method} for phylogenetic tree reconstruction is \rot{a map which takes a dissimilarity metric as input data, and maps} onto the space of phylogenetic trees.
\edit{Distance-based reconstruction methods were generalized to Bergman fans of arbitrary matroids by  \cite{ardila2005subdominant} via tropical projection, and studied by  \cite{bernstein2020infinity}.}
\edit{There, the points of the Bergman fan of $M$ are called $M$-ultrametrics.}

\edit{In this paper,} we study the tropical projection of {\em tropical Fermat-Weber points} to spaces of $M$-ultrametrics via the map defined in \Cref{defn: tropical projection}.
\edit{Fermat-Weber points are a generalization of the median to arbitrary metric spaces.
{\em Tropical} Fermat-Weber points are Fermat-Weber points in the tropical projective torus, equipped with the standard tropical metric.}
\edit{It is known that in this case,} even if all observations are \edit{$M$-ultrametrics}, the \edit{tropical} Fermat-Weber points are not, in general, $M$-ultrametrics \citep{lin2017convexity}\edit{; for more on tropical Fermat-Weber points, see \cite{sabol2024polytropes}.}
\edit{In \Cref{thm: symmetric tropical median}, we show} that if all observations are \edit{$M$-ultrametrics, then} the intersection \rot{of} the set of \edit{tropical} Fermat-Weber points and \edit{the space of $M$-ultrametrics} is not empty, and that it is equal to the projection of the Fermat-Weber points onto \edit{the space of $M$-ultrametrics}.  

In practice, \edit{observed data} are usually not \edit{$M$-}ultrametrics \rot{due to noise in the observations}. 
\rot{To quantify the maximum error under which projected \edit{tropical} Fermat-Weber points are stable}, we extend the \edit{definition of} {\em safety radius} \edit{for phylogenetic reconstruction, originally given} by \edit{Atteson in} \citep{atteson1999performance}, to \edit{the} projection of the set of all \edit{tropical} Fermat-Weber points to \edit{the space of $M$-ultrametrics}.  
\edit{In \Cref{thm: tropical safety radius}, w}e show that \edit{this $M$-ultrametric reconstruction method has} safety radius $1/2$ for any matroid\edit{, which is the maximum possible}.  

We also \rot{consider the influence of stochastic error, for which we} define the {\em Fermat-Weber $M$-ultrametric stochastic safety radius} \edit{(\Cref{defn: stochastic safety radius}), extending the definition of stochastic safety radius for phylogenetic construction given by Gascuel and Steel in \citep{gascuel2016stochastic}.}
\edit{In \Cref{thm: stochastic safety radius}, we show that our generalized stochastic safety radius is always non-zero.} 
\edit{We end with some computational experiments demonstrating the $\fw$ $M$-ultrametric safety radius for a sample from a particular Bergman fan under additive Gaussian noise, and a summary of experiments for all matroids on 8 elements. 
Our results suggest that in practice our bound for the $\fw$ $M$-ultrametric safety radius is very conservative, which aligns with previous observations in the case of phylogenetic trees \citep{gascuel2016stochastic}.}

\edit{The paper is structured as follows. \Cref{sec:safety} reviews the necessary background in phylogenetics and tropical geometry, and connects these areas through the definition of the $M$-ultrametric safety radius (\Cref{defn: safety radius}). In \Cref{sec:FW_Safety}, we determine the safety radius of the procedure which takes a finite sample of $M$-ultrametrics to the projection of its tropical Fermat-Weber points. Finally, in \Cref{sec:stochastic safety radius}, we prove bounds on the stochastic safety radius and discuss experimental results.}

\edit{
\subsection*{Notation}
Throughout this paper we denote $[p]:=\{1, \ldots , p\}$ for a positive integer $p$. We use $\qq_U$ to denote the zero-one vector with support $U$, and write $x_a + y_a$ as $(x+y)_a$. For a point $w \in \R^q$, we use the notation $\argmax_{e \in [q]} w_e = \{ e \in [q] : w_e = \max_{f \in [q]} w_f \}$. We use $\mathbb{I}_q$ to denote the $q\times q$ identity matrix, and $\mathbb{P}(A)$ for the probability of event $A$.
}

\section{Safety Radius for $M$-Ultrametrics}
\label{sec:safety}

In this section, \edit{we recall} the tropical projection onto the space of \edit{$M$-}ultrametrics, and \edit{we define} its {\em safety radius}.
We begin by reminding readers of some background on phylogenetic trees.

\subsection{Phylogenetics}

A phylogenetic tree on $p$ leaves is a \rot{rooted} tree with non-negatively weighted edges which represents the evolutionary history of a set of $p$ taxa or species.
Its internal vertices are unlabeled\edit{,} and its $p$ leaves are labeled \edit{bijectively} by the elements of $[p]$. 
We further assume \edit{the root has degree at least two, and all other internal vertices have degree at least three}.
\edit{We denote by $E(T)$ the edges of the phylogenetic tree $T$.}
An edge of $T$ is an \emph{internal edge} if neither endpoint is a leaf.
\rot{The \emph{space of all phylogenetic trees on $p$ leaves} is denoted $\tspace_p$.}

The \emph{topology} of a phylogenetic tree $T$, which we denote by $\tau(T)$, is the underlying unweighted, leaf-labeled graph of $T$.
A tree topology is \emph{binary} if the root has degree two, and every other non-leaf vertex has degree three.
For example, \rot{of the 26 distinct tree topologies with four leaves, 15 topologies are binary.}

A \emph{distance-based phylogenetic reconstruction method} (or \emph{distance-based method}) takes as input estimated pairwise evolutionary distances among $p$ taxa, and outputs a phylogenetic tree on $p$ leaves. 
\edit{For further information on how these distances are obtained in practice, see} \citep{phylogeneticsTextbook}.  
\edit{This} set of estimated pairwise distances \rot{forms a \textit{dissimilarity matrix}, that is, a symmetric non-negative $p \times p$ matrix, $\edit{D=[d_{ij}]}$, \edit{where $d_{ij}$} encodes the evolutionary distance between taxa $i,j \in [p]$}.
Note that $D$ is not necessarily a metric, since we do not assume that it satisfies the triangle inequality.
We denote the \emph{set of \edit{$p \times p$} dissimilarity matrices} by $\dspace_p$.

The \emph{tree metric} associated \edit{to} a phylogenetic tree $T$ is \rot{the dissimilarity matrix whose $(i,j)$ entry records the length of the unique path in $T$ from leaf $i$ to leaf $j$}.
\edit{Each tree metric arises from a unique tree \cite[Theorem 7.2.8]{phylogeneticsTextbook}. 
Thus, we identify the space of phylogenetic trees $\mathcal T_p$ with the subset of dissimilarity matrices that arise as tree metrics}.
With this notation, a distance-based phylogenetic reconstruction method is simply a function $\phi: \dspace_p \to \tspace_p$. 

An \emph{equidistant tree} on $p$ leaves is a phylogenetic tree such that the distance from the root to any leaf is the same. 
\rot{Such} trees appear \edit{in phylogenetics} under the molecular clock assumption, which asserts that all lineages evolve at roughly equal rates.
The tree metric of an equidistant tree is also known as an ultrametric \citep{buneman1974properties}.
\edit{We use$\,\,\uspace_p$ to refer to the space of equidistant trees on $p$ leaves, or equivalently to the subset of $p \times p$ dissimilarity matrices that arise as the tree metric of an equidistant tree.}

\edit{In practice, dissimilarity measures are estimated from data which are noisy or incomplete, resulting in dissimilarity matrices that typically do not satisfy the four point condition.}
While it is impossible to recover \edit{a tree metric exactly from its noisy dissimilarity matrix}, the primary concern of phylogenetic reconstruction is \edit{recovering} the correct tree topology. 
Therefore, one asks: when does a distance-based method produce a tree with the same topology as $T$, given noisy input $D = D_T + \epsilon$~?

The \emph{safety radius}, introduced by \cite{atteson1999performance}, quantifies the amount of noise a distance-based method can tolerate.
The tolerance for a given tree $T$ depends on the minimum edge length in $T$.

\begin{definition}[Safety radius, {\citep{atteson1999performance}}]
    \label{defn: safety radius trees}
    Let $l(e)$ denote the length of edge $e$ for $e \in \edit{E(T)}$. A distance-based reconstruction method $\phi: \dspace_p \to \tspace_p$ has \emph{safety radius} $s$ if for every binary tree $T \in \tspace_p$, and every dissimilarity matrix $D \in \dspace_p$ such that
    \[ \| D_T - D \|_\infty ~ < ~s ~\cdot \min_{e \in \edit{E^(T)}} l(e), \]
    the tree $\phi(D)$ has topology $\tau(T)$.
\end{definition}

\edit{Atteson} also \rot{ shows that the safety radius of any distance-based method is at most $1/2$ ({Lemma 3, \cite{atteson1999performance}})}.
\edit{In \Cref{defn: safety radius}, we will adapt this safety radius to ultrametrics of arbitrary matroids.}

\subsection{$M$-Ultrametrics}

Our goal in this section is to generalize the \edit{notion of} safety radius to $M$-ultrametric reconstruction methods. An \emph{$M$-ultrametric reconstruction method} is a function that takes as input \edit{a point in $\TPT{q}$} and outputs an $M$-ultrametric.
The reconstruction method we study is a nearest point map under the tropical symmetric metric to \edit{the space of $M$-ultrametrics}.
We begin with the relevant background on matroids and tropical geometry. 
We assume that the reader has some familiarity with these topics; for further details, we refer the reader to \citep{ETC}.

$M$-ultrametrics are a generalization of ultrametrics \edit{to arbitrary} matroids \edit{as} introduced by \cite{ardila2005subdominant}. \edit{We refer to the set of all $M$-ultrametrics for the matroid $M$ as the {\em space of $M$-ultrametrics}.}

\begin{definition}
    \label{defn: bergman fan}
    Let $M$ be a matroid on ground set $[q]$. 
    The \edit{\emph{space of $M$-ultrametrics}}, denoted $\uspace_M$, is the set of all $w \in \R^q$ such that for every circuit $C$, $\argmax_{e \in C} w_e$ has cardinality at least two.
\end{definition}

\rot{The space of $M$-ultrametrics can be endowed with the structure of a polyhedral fan, the coarsest such fan structure being the Bergman fan. Since $\uspace_M$ is the underlying set of the Bergman fan $\berg(M)$, we sometimes refer to the set $\uspace_M$ as a Bergman fan.}
\edit{The fan structure we consider on $\mathcal{U}_M$ is the nested set fan on connected flats as described in \citep{bernstein2020infinity}; see also \citep{feichtner2005matroid}.}
We denote this fan structure by $\nested(M)$. \rot{Through the rest of this paper, we make the usual assumption that $M$ is connected for notational ease.}

Recall that a \emph{nested set}, $\nset$, is a set of connected flats such that for all pairwise incomparable $F_1, \ldots, F_r \in \nset$, the closure of their union, $\overline{F_1 \cup \cdots \cup F_r}$, is disconnected. 
\edit{As in \citep{bernstein2020infinity}, we further require that every nested set contains $[q]$.}
The cones of $\nested(M)$ are in bijection with \rot{the collection of} nested sets, and we denote the cone corresponding to \rot{a nested set} $\nset$ by $K_\nset$.
For each $\nset$, we have
\begin{align}\label{eq:K_S}
    K_\nset &\edit{\coloneqq} \{ \mu \qq + \textstyle\sum_{F \in \nset} -\lambda_F \qq_F | \lambda_F \geq 0, \mu \in \R \}.
    %\relint K_\nset &= \{ \mu \qq + \sum_{F \in \nset} -\lambda_F \qq_F | \lambda_F > 0, \mu \in \R \}.
\end{align}
\edit{We will denote by }$\relint K_\nset$ the relative interior of $K_\nset$, which is the interior of $K_\nset$ as a subset of the affine span of $K_\nset$.
The boundary of $K_\nset$ is $\partial K_\nset \coloneqq K_\nset \setminus \relint K_\nset$.

In this paper, we  use an alternative characterization of the cones of $\nested(M)$, which is described in \rot{terms of the $\argmax$ sets for each circuit}. \rot{The {\em argmax fan} structure on $\mathcal{U}_M$ consists of the cones $\{ K^w$ : $w \in \mathcal{U}_M \}$, where:
\[
    K^w := \overline{\{ u \in \mathcal{U}_M : \text{ for all circuits $C$ of $M$,} \argmax_{e \in C} u_e = \argmax_{e \in C} w_e \}}.
\]
It is clear that each $K^w$ is a cone, and $K^w \subseteq K^{w'}$ if and only if, for all circuits $C$ of $M$, $\argmax_{e \in C} w'_e \subseteq \argmax_{e \in C} w_e$, with a strict inclusion on the left if and only if there is a strict inclusion on the right for at least one circuit. We now show that the argmax fan is exactly the nested set fan.} 

\begin{lemma}
    \label{lem: nested argmax}
    \rot{Let $K_{\mathcal S}$ be a cone in the nested set fan $\nested(M)$ with $w \in \relint K_{\mathcal S}$. Then $K_{\mathcal S} = K^w$. Furthermore, if $K_{\mathcal S'}$ is another cone in the nested set fan $\nested(M)$ with $w' \in \relint K_{\mathcal S'}$, then}
    \begin{equation}
        \label{eqn: bergman inclusion}
        \rot{K_{\mathcal S} \subseteq K_{\mathcal S'} \iff \rot{K^{w} \subseteq K^{w'}}.}
    \end{equation}
\end{lemma}
\begin{proof}
    Given any nested set $\nset$, and any $w \in \relint K_\nset$, the data $\{ (C, \argmax_{e \in C} w_e) : C \text{ is a circuit of } M \}$ is determined by $\nset$.
    Specifically, we claim that for any circuit $C$, $\argmax_{e \in C} w_e = C \setminus \cup_{C \not\subset F \in \nset} F =: C^\prime$.
    
    First, we show that $C^\prime$ is not empty. 
    Suppose otherwise, and let $F_1, \dots, F_m$ be a set of connected flats covering $C$, and satisfying $C \not\subset F \in \mathcal S$.
    Without loss of generality, we may assume that $F_1, \ldots, F_m$ are pairwise incomparable. 
    Let $F = \overline{F_1 \cup \dots \cup F_m}$.
    Since each flat is connected and intersects $C$ non-trivially, $F$ is connected.
    But this contradicts the assumption that $\nset$ is a nested set.
    Therefore, $C^\prime$ is non-empty.
    Moreover, $\argmax_{e \in C} w_e = C^\prime$.
    To see this, recall that $w = \sum_{F \in \edit{\mathcal S}} -\lambda_F \qq_F$, with $\lambda_F > 0$ for all $F$ \edit{(as $w$ is in the relative interior of $K_\nset$)}.
    For all $e, e^\prime \in C^\prime$ and $f \in C \setminus C^\prime$, $w_e = w_{e^\prime} > w_f$.
    Thus, $\argmax_{e \in C} w_e = C^\prime$.

    Now \edit{consider} $w \in \rot{\mathcal{U}_M}$.
    We will show there is a nested set $\nset$ so that $\rot{K^{w}} = K_\nset$.
    We just showed that $\nested(M)$ refines the \edit{argmax} fan structure.
    Thus, \edit{if $K^{w}$ is not a cone of $\nested(M)$}, there are two nested sets $\nset_1, \nset_2$ so that $K_{\nset_1}, K_{\nset_2} \subseteq \rot{K^{w}}$, and $\relint (K_{\nset_1} \cap K_{\nset_2}) \subseteq \relint \rot{K^{w}}$.
    Let $w^1 \in \relint K_{\nset_1}$, and $w^2 \in \relint K_{\nset_1 \cap \nset_2} = \relint (K_{\nset_1} \cap K_{\nset_2})$.
    Since $w^1, w^2 \in \relint \rot{K^{w}}$, we have $\argmax_{e \in C} w^1_e = \argmax_{e \in C} w^2_e, \text{ for all circuits } C \text{ of } M.$

    Suppose for a contradiction that $\nset_1 \neq (\nset_1 \cap \nset_2)$.
    Then there exists $F \in \nset_1 \setminus (\nset_1 \cap \nset_2)$.
    As \edit{every nested set contains $[q]$,} $\edit{\mathcal S_1} \cap \edit{\mathcal S_2}$ is a \edit{non-empty} nested set\edit{.
    Thus,} there must exist some minimal $F' \in \edit{\mathcal S_1} \cap \edit{\mathcal S_2}$ \edit{containing} $F$. Note that as $F \notin \edit{\mathcal S_2}$, we have the strict containment $F \subset F'$ and we can pick some $i \in F' \setminus F$.

    We now define $F_1, \dots, F_k$ to be the maximal flats in $\edit{\mathcal S_1} \cap \edit{\mathcal S_2}$ which are contained in $F$. Note that as these flats are incomparable, the closure of their union cannot be connected (by the definition of nested sets). This closure must therefore be strictly contained in $F$, so we can pick some $j$ in $F \setminus (F_1 \cup \dots \cup F_k)$.

     The flat $F'$ is connected, so there is some circuit $C \subset F'$ connecting $i,j$. We consider the $\argmax$ over this circuit. For $w^1 \in \relint K_{\edit{\mathcal S_1}}$, we have shown that:
    \begin{align*}
        \argmax_{e \in C} w^1_e = C \setminus \cup_{C \not\subset F'' \in \mathcal{S}^1} F'' \subseteq C \setminus F.
    \end{align*}
    Hence $j \notin \argmax_{e \in C} w^1_e$.

    We now consider $\argmax_{e \in C} w^2$ for $w^2 \in \edit{\relint} K_{\edit{\mathcal S_1} \cap \edit{\mathcal S_2}}$.
    We claim there is no element of $\nset_1 \cap \nset_2$ that both contains $j$ and does not contain $C$.
    If our claim holds, then
    \begin{align*}
        \argmax_{e \in C} w^2_e = C \setminus \cup_{C \not\subset F'' \in \edit{\mathcal{S}_1} \cap \edit{\mathcal S_2}} F''  \ni j.
    \end{align*}
    Therefore we have some circuit $C$ such that $\argmax_{e \in C} w^1_e \neq \argmax_{e \in C} w^2_e$, and this contradicts our assumption that $w^1, w^2 \in \relint \rot{K^{w}}$.

    \rot{It therefore remains to prove our claim that there is no element of $\nset_1 \cap \nset_2$ that both contains $j$ and does not contain $C$;} suppose for a contradiction that $F''$ is an element of $\nset_1 \cap \nset_2$ such that $C \not \subset F''$ and $j \in F''$. Then $F'' \subset F'$ as $F''$ has non-empty intersection with $F'$ but doesn't contain all of $C$. Also, either $F \subset F''$ or $F'' \subset F$ as we have $j \in F \cap F''$. Since no flat satisfies $F \subset F'' \subset F'$, we must have $F'' \subset F$. But our choice of $j$ ensures that $j \notin F''$; thus, no such $F''$ exists. Hence, for all $F''$ satisfying $C \not \subset F'' \in \edit{\mathcal S_1} \cap \edit{\mathcal S_2}$, either $C \cap F'' = \emptyset$ or $j \notin F''$.

    \rot{We have shown that the cones of the nested set fan and the argmax fan are in bijection. 
    We conclude by noting that if $w^\prime$ is a point on the boundary of $K^w$}, then $w \in K_\nset$ and $w^\prime \in K_{\nset^\prime}$, where $\nset^\prime \subseteq \nset$.
    Therefore, $\argmax_{e \in C} w^\prime_e \subseteq \argmax_{e \in C} w_e$ for all circuits $C$, \rot{and therefore $K^{w'} \subseteq K^w$}.
    This completes the proof of (\ref{eqn: bergman inclusion}).
\end{proof}

When \rot{$M(K_p)$} is the matroid of the complete graph on $p$ vertices, \rot{$\uspace_{M(K_p)}$} is the space of equidistant trees on $p$ leaves \citep{arilda2006bergmancomplex}\rot{, or equivalently the space of ultrametrics $\mathcal U_p$}.
In this case, the coordinates of the ambient space, $\R^{\edit{q}}$, correspond to the \edit{$q = {p \choose 2}$} edges of the complete graph.
Each cone of $\nested(M)$ corresponds to a unique tree topology on $p$ leaves. 
\rot{This is why points in $\uspace_M$ are referred to as $M$-ultrametrics.} 
\edit{For general matroids, the fan structures $\berg(M)$ and $\nested(M)$ may not coincide, as demonstrated in the following example.}

\begin{example}\label{ex:refinement}
    \edit{Let $M=M(G)$ be the rank $r=5$ graphical matroid described by the graph $G$ in \cref{fig:M(G)}, which will serve as our running example. $M$ has 30 bases, 6 circuits, and 78 (proper) flats. $\berg(M)$ contains 54 maximal cones, 18 of which are further refined by $\nested(M)$, which contains 79 maximal cones. 
    Further refinement is possible, e.g.\ via Rincon's \emph{cyclic} Bergman fan \citep{rincon2013computing}, which contains 88 maximal cones and subdivides 9 maximal cones of $\nested(M)$.}
    
    \edit{
    Maximal cones of the Bergman fan $\berg(M)$ correspond to unordered partitions $\{B_1, \ldots , B_r\}$ of $[q]$ into $r$ non-empty blocks \cite[Corollary 4.7]{feichtner2005matroid}.
    Let $K \in \berg(M)$ be the cone identified by the partition $\{ \{1,2 \}, \{3 \}, \{4,5 \}, \{6 \}, \{7,8 \} \}$.
    \cref{fig:M(G)} depicts the two nested sets $\mathcal{S}_1$ and  $\mathcal{S}_2$ of the partition, which refine $K$ into $K_{S_1}, K_{S_2} \in \nested(M)$. Let $F_{U}:=\{e_j \mid j\in U\}$ denote a flat of $M$. For $\mathcal{S}_1=\{F_{3}, F_{6}, F_{378}, F_{456} \}$, if we let $\mu=5$, $\lambda_{F_3}=\lambda_{F_6}=\lambda_{F_{378}}=2$, and $\lambda_{F_{456}}=3$, then \cref{eq:K_S} yields the vector $w=(5,5,1,2,2,0,3,3) \in \relint(K_{S_1})$. The cone $K_{S_1}$ is further refined into two maximal cones in the cyclic Bergman fan through $F_{345678}$, which is cyclic (see \citep{rincon2013computing}) but not connected.
    }
\end{example}

\begin{figure}
    \centering 
    \scriptsize
    \begin{tikzpicture}[scale=2]
        \tikzstyle{vertex}=[circle, draw, inner sep=2pt]
      % Nodes
      \node[vertex] (a) at (0,0) {};
      \node[vertex] (b) at (1,0) {};
      \node[vertex] (c) at (1,1) {};
      \node[vertex] (d) at (0,1) {};
      \node[vertex] (e) at (-0.5,0.5) {};
      \node[vertex] (f) at (0.5,0.5) {};
      % Edges
      \draw (a) -- node[above] {1} (b);
      \draw (b) -- node[right] {2} (c);
      \draw (c) -- node[below] {3} (d);
      \draw (a) -- node[left] {4} (d);
      \draw (a) -- node[left] {5} (e);
      \draw (d) -- node[above left] {6} (e);
      \draw (c) -- node[below right] {7} (f);
      \draw (d) -- node[below left] {8} (f);
      % Nested Set
      \node (12) at (1.75,0.5) {$\{1,2\}$};
      \node (45) at (2.25,0.9) {$\{4,5\}$};
      \node (78) at (2.25,0.1) {$\{7,8\}$};
      \node (3) at (2.9, 0.1) {$\{3\}$};
      \node (6) at (2.9, 0.9) {$\{6\}$};
      \draw[->] (12) -- (45);
      \draw[->] (12) -- (78);
      \draw[->] (45) -- (6);
      \draw[->] (78) -- (3);

      \node (12b) at (4.4,0.5) {$\{1,2\}$};
      \node (45b) at (4.9,0.9) {$\{4,5\}$};
      \node (78b) at (3.75,0.5) {$\{7,8\}$};
      \node (3b) at (4.9, 0.1) {$\{3\}$};
      \node (6b) at (5.55, 0.9) {$\{6\}$};
      \draw[->] (12b) -- (3b);
      \draw[->] (12b) -- (45b);
      \draw[->] (45b) -- (6b);
      \draw[->] (78b) -- (12b);
    \end{tikzpicture}
\caption{\edit{Left: The graph $G$ representing the graphical matroid $M(G)$ in \cref{ex:refinement}. Center: The nested set structure of the flats generating the maximal cone $K_{S_1}\in \nested(M)$. Right: The nested set structure for $K_{S_2}$.}}
    \label{fig:M(G)}
\end{figure}

We view the \rot{space of $M$-ultrametrics} through the lens of tropical geometry.
The \emph{tropical max-plus semiring} is $\T := (\R \cup \{- \infty\}, \oplus, \odot)$, where $a \oplus b$ is defined as the maximum, $\max(a, b)$, and $a \odot b$ \edit{is} given by the classical addition $a + b$.
Tropical arithmetic extends coordinate-wise to the semimodule $\T^q$. 
The tropical sum of two vectors is their coordinate-wise maximum, and tropical scalar multiplication by $\lambda \in \T$ is translation by $\lambda \qq$.

The $(q-1)$-dimensional \emph{tropical projective torus}, $\TPT{q}$, is a quotient space constructed by endowing $\R^{q}$ with the equivalence relation $x \sim y \Leftrightarrow x = a \odot y$. Similarly, the $(q-1)$-dimensional \emph{tropical projective space} is $\TP^{q-1} = (\T^{q} \setminus \{ -\infty \qq \})/\R\qq$.
\edit{Going forward, we identify $\uspace_M$ with its image in $\TPT{q}$; for the closure of $\uspace_M$ in $\TP^{q-1}$, we use $\tilde{\uspace}_M$.}
This is justified by the fact that for any $M$-ultrametric $w$, $a \odot w$ is also an $M$-ultrametric for any $a \in \T$.

Ardila showed in \cite[Section 4]{ardila2005subdominant} that $\tilde\uspace_M$ is a tropical polytope in $\TP^{q-1}$.
A \emph{tropical polytope} \rot{$P$} is the set of all tropical linear combinations of points in a finite set $V \subseteq \TP^{q-1}$.
If $V = \{ v_1, \ldots, v_r \}$, we denote this set by
\[ \tconv(V) = \{ (\lambda_1 \odot v_1) \oplus \dots \oplus (\lambda_r \odot v_r) : \lambda_1, \ldots, \lambda_r \in \T \}. \]

%% vertices of Bergman fan
\begin{proposition}[Section 4, {\citep{ardila2005subdominant}}]
    \label{prop: bergman vertices}
    The vertices of $~\edit{\tilde\uspace_M}$ as a tropical polytope are 
    \[ V_M = \{ v_F : F \text{ is a maximal proper flat of } M\}, \]
    where $v_F \in \TP^{q-1}$ is the vector whose $\ith$ coordinate is $-\infty$ if $i \in F$, and 0 otherwise. 
\end{proposition}

\rot{We now consider the metric geometry of the tropical projective torus and Bergman fans.}
\rot{We endow} the tropical projective torus \rot{with} the \emph{tropical metric}, which is given by
\begin{equation}\label{eq:sym_dist}
    \dtr(x,y) = \max_{1 \leq i \leq q} \{ y_i-x_i \} - \min_{1 \leq i \leq q} \{ y_i - x_i \} = \max_{1 \leq i, j \leq q} \{ x_i - y_i - x_j + y_j\}
\end{equation}
\noindent for all $x, y \in \TPT{q}$.

 \cite{develin2004tropicalconvexity} give an explicit map, $\pi_P : \TP^{q-1} \to P$, which outputs a nearest point in $P$ under the tropical metric.
We recall that map now.

\begin{definition}[Projection to a Tropical Polytope {\citep{develin2004tropicalconvexity}}]
    \label{defn: tropical projection}
    Let $P = \tconv(v_1, \ldots, v_r)$.
    For $i \in [r]$, define $\lambda_i = \max \{ \lambda \in \R : \lambda \odot v_i \odot x = x \} = \min_e (x - v_i)_e$.
    For any $x \in \TP^{q-1}$, the point
    \begin{equation*}
        \label{eqn: tropical projection}
        \pi_P(x) = (\lambda_1 \odot v_1) \oplus \cdots \oplus (\lambda_r \odot v_r)
    \end{equation*}
    minimizes the distance $\dtr(y,x)$ over all $y \in P$.
    We call $\pi_P(x)$ the tropical projection to $P$.
\end{definition}

\edit{The vertices of $\uspace_M$ given in \Cref{prop: bergman vertices}, together with the tropical projection in \Cref{defn: tropical projection}, define a map $\pi_{\edit{\tilde{\uspace}_M}}: \TP^{q-1} \to \edit{\tilde{\uspace}_M}$.}
\edit{Restricting the domain to $\TPT{q}$ yields a map} $\pi_{\edit{\uspace_M}}: \edit{\TPT{q}} \to \edit{\uspace_M}$\edit{, which we refer to} by $\pi_M$ to emphasize its dependency on the matroid $M$.

A particularly important fact that we will use about tropical projections is that they are non-expansive under the tropical metric.
To our knowledge, this result first appeared in the preprint \citep{jaggi2008new}; we also include a short proof here.

 \begin{lemma}[Lemma 19, {\citep{jaggi2008new}}]
    \label{lem: non-expansive}
    Let $\pi_M: \edit{\TPT{q}} \to \edit{\uspace_M}$ be the tropical projection to the \edit{space of $M$-ultrametrics}. 
    For all $x, y \in \edit{\TPT{q}}$,
    \begin{equation}
        \label{eqn: non-expansive}
        \dtr(\pi_M(x), \pi_M(y)) \leq \dtr(x, y).
    \end{equation}
\end{lemma}

\begin{proof}
    Recall that $V_M$ is the vertex set of $\edit{\uspace_M}$ as a tropical polytope \rot{(\Cref{prop: bergman vertices})}.
    Let $\lambda^x_i, \lambda^y_i$ be the coefficients of $v_i \in V_M$ in the tropical expressions for $\pi_M(x), \pi_M(y)$. 
    Fix vertices $v_k,v_l \in V_M$ such that $\pi_M(x)_j = \lambda_k^x + v_{kj} = \max_{i} \lambda_i^x + v_{ij}$ and $\pi_M(y)_j = \lambda^y_l + v_{lj} = \max_{i} \lambda_i^y + v_{ij}$.
    Then
    \begin{align*}
        \pi_M(x)_j - \pi_M(y)_j &= \lambda_k^x + v_{kj} - \lambda^y_l - v_{lj}, \\
                &\geq \lambda^x_l + v_{lj} - \lambda^y_l - v_{lj}, \\
                &\geq \textstyle \min_e (x_e-v_{le}) - \textstyle \min_f (y_f-v_{lf}), \\
                &\geq \textstyle \min_e (x_e-v_{le} - y_e + v_{le}), \\
                &= \textstyle \min_e (x_e- y_e).
    \end{align*}
    \edit{Similarly}, we can show that:
    \[
        \pi_M(x)_j - \pi_M(y)_j \leq \textstyle \max_e (x_e - y_e).
    \]
    Hence we conclude that $\dtr(\pi_M(x), \pi_M(y)) \leq \dtr(x, y).$
\end{proof}

\subsection{Tropical Projection Safety Radius}
\label{subsec:proj_safety_radius}

Here we \edit{generalize} the safety radius \rot{for phylogenetic reconstruction \edit{defined in \citep{atteson1999performance}}} to \edit{the projection $\pi_M$ to the} Bergman fan \edit{of an arbitrary matroid}.
\rot{The classical safety radius measures how much a tree metric can be perturbed, assuming that the phylogenetic reconstruction method recovers the same tree topology.}
The topology of $T$ is encoded by \rot{the cone of $\mathcal T_p$ which contains $D_T$ in its relative interior.}
\edit{Thus, the question is: for $w \in \uspace_M$, when is $\pi_M(w + \epsilon)$ contained in the relative interior of $K^w$?}

The tolerance for noise depends on how close the input $w \in \edit{\uspace_M}$ is to the boundary of $K^w$.
To measure this distance, we use the following notation.
For a set $\edit{U}$, the \emph{first maximum of $\edit{U}$}, denoted $\max^1 \edit{U}$, is the usual maximum, while the \emph{second maximum of $\edit{U}$}, denoted $\max^2 \edit{U}$, is the second greatest value in $\edit{U}$.
When \edit{we} write $\max^i w$ \edit{for} $w \in \R^q$, \edit{we} mean the $i^{th}$ maximum of the entries of $w$ as a set.
In particular, for non-constant $w$, $\max^1 w$ is always strictly greater than $\max^2 w$.
Similarly, we use $\min^i$ to denote the $\ith$ minimum.

\begin{lemma}
    \label{lemma: dist to cone boundary}
    Let $M$ be any matroid, and let $K$ be a maximal cone in the \edit{nested set} fan $\nested(M)$. For any $w \in \relint K$, the tropical distance from $w$ to the boundary $\partial K$ is $\wmin \coloneqq \min_{C} \left( \max^1_{e\in C} w_e - \max^2_{e \in C} w_e \right)$, where this minimum is taken over all $C$ such that $\max^1_{e\in C} w_e - \max^2_{e \in C} w_e \neq 0$.
\end{lemma}

As $w$ is in the relative interior of a maximal cone, it is not the constant vector. We also know that $M$ is connected, so there must be some circuit on which $C$ is not constant. This ensures that $\wmin$ is well defined and positive. 

\begin{proof} 
    Let $u$ be any point in $\edit{\uspace_M}$, and let $K^\prime$ be the cone of $\nested(M)$ containing $u$ in its relative interior.
    Recall that $K \subseteq K^\prime$ if and only if for all circuits $C$, $\argmax_{e \in C} u_e \subseteq \argmax_{e \in C} w_e$.
    Since $K$ is a maximal cone, this is equivalent to $K = K^\prime$.
    
    We first show that if $\dtr(u,w) < \wmin$, then $u \in \relint K$.
    We proceed by contradiction.
    Suppose there is a circuit $C$ of $M$ and an index $a \in [q]$ such that $u_a$ is maximal for $C$, but $w_a$ is not maximal for $C$. 
    Let $b \in \argmax_{e \in C} w_e$. As $w$ is not constant on $C$, we have that $$\wmin \leq \textstyle\max^1_{e\in C} w_e - \textstyle\max^2_{e \in C} w_e \leq w_b - w_a.$$ \edit{Since} $\dtr(u,w) < \wmin$:
    \begin{align*}
        u_a - u_b &< w_a -w_b + \wmin \leq w_a -w_b + (w_b - w_a) = 0.        
    \end{align*}
    This implies $u_a < u_b$, contradicting the maximality of $u_a$. 
    Thus, we conclude $u \in \relint  K$.

    It remains to show that there is some $u \in \partial K$ such that $\dtr(u,w) = \wmin$. 
    We construct such a $u$ explicitly. 
    Define $C$ to be a circuit of $M$ such that $\textstyle\max^1_{e\in C} w_e - \textstyle\max^2_{e \in C} w_e = \wmin$, which exists as $w$ is not constant and $M$ is connected. Let $B$ be the set of all indices $b \in [q]$ for which $w_b = \max^2_{e \in C} w_e$.
    We define $u := w + \wmin \qq_B$.
    For an example of this construction with equidistant trees, see \Cref{fig:tree_dist_to_boundary}.

    By construction, $\dtr(w,u) = \wmin$. 
    We will now show that $u \in \partial K$. 
    For each circuit $C$ of $M$, we have four cases to consider:
    \begin{enumerate}
        \item $B \cap C = \emptyset$,
        \item $B \cap C \neq \emptyset, \max_{e \in C} w_e \geq w_b + \wmin$,
        \item $B \cap C \neq \emptyset, w_b < \max_{e \in C} w_e < w_b + \wmin$,
        \item $B \cap C \neq \emptyset, w_b = \max_{e \in C} w_e$.
    \end{enumerate}
    In case (1), $B \cap C = \emptyset$, and thus $u_e = w_e$ for all $e \in C$. In particular, that implies $\argmax_{e \in C} u_e = \argmax_{e \in C} w_e$.
    In case (2), $\max_{e \in C} w_e$ is unchanged, although more coordinates may achieve the maximum. Thus, $\argmax_{e \in C} u_e \supseteq \argmax_{e \in C} w_e$.
    Case (3) is not possible, since that would imply $\wmin > \max_{e \in C} w_e - w_b$, and that contradicts the definition of $\wmin$.
    Finally, in case (4), $B \cap C = \argmax_{e \in C} w_e$. Thus, the maximum increases: $\max_{e \in C} u_e = \max_{e \in C} w_e + \wmin$; but the coordinates achieving the maximum remain the same: $\argmax_{e \in C} u_e = \argmax_{e \in C} w_e$.

    Hence, for every circuit we have $\argmax_{e \in C} u_e \supseteq \argmax_{e \in C} w_e$.
    Therefore, $u \in K$.
    However, for the circuit $C^\prime$, the containment $\argmax_{e \in C^\prime} u_e \supset \argmax_{e \in C^\prime} w_e$ is strict, so $u \notin \relint K$.
    It follows that $u \in \partial K$.
\end{proof}

\begin{figure}
    \centering
  \begin{subfigure}[b]{.33\textwidth}
    \centering
        \begin{tikzpicture}
            \node {root} [level distance = 1cm, sibling distance = 1cm]
                child {node [yshift = -2cm, xshift = -1cm] {$A$} edge from parent} 
                child {node {}
                    child {node [yshift = -1cm, xshift = -0.5cm] {$D$}}
                    child {node {}
                        child {node {$B$}}
                        child {node {$C$}}
                        edge from parent node[right] {$l$}
                        }
                    edge from parent
                };
        \end{tikzpicture}
      \end{subfigure}
  \begin{subfigure}[b]{.33\textwidth}
    \centering
        \begin{tikzpicture}
            \node {root} [level distance = 1cm, sibling distance = 1cm]
                child {node [yshift = -2cm, xshift = -1cm] {$A$} edge from parent} 
                child {node {}
                    child {node [yshift = -1cm] {$D$}}
                    child {node [yshift = -1cm] {$B$}}
                    child {node [yshift = -1cm] {$C$}}
                    edge from parent
                };
        \end{tikzpicture}
  \end{subfigure}
    \caption{\edit{When $M=M(K_p)$ is the matroid of the complete graph on $p$ vertices,  $w \in \uspace_M$ corresponds to an equidistant tree. When $w$ is in the relative interior of a maximal cone of $\nested(M)$, the corresponding tree is binary, as exemplified by the tree on the left. In this case, the value $\wmin$ is exactly the tree's minimal internal edge length, $l$.} The right equidistant non-binary tree has metric $u$, which has tropical distance $l$ from $w$, as constructed in the proof of \cref{lemma: dist to cone boundary}.}
    \label{fig:tree_dist_to_boundary}
\end{figure}

\Cref{lemma: dist to cone boundary} motivates the following definition of safety radius for $M$-ultrametrics.

\begin{definition}[Safety Radius]
    \label{defn: safety radius}
    \edit{Let $\pi_M: \TPT{q} \to \edit{\uspace_M}$.} The safety radius of \edit{$\pi_M$} is the largest $s \in \R_{\geq 0}$ such that for every $w$ in the relative interior of a maximal cone $K$ of \edit{$\nested (M)$}:
    \[ \pi_M(w + \epsilon) \text{ is in the relative interior of }K \text{ for every } \epsilon \in (-s \wmin,s \wmin)^{q}, \] 
    where 
    \edit{
    $$\wmin = \min_{C} \left( \mathrm{max}^1_{e \in C} w_e - \mathrm{max}^2_{e \in C} w_e \right).$$
    }
\end{definition}

In introducing the safety radius of phylogenetic reconstruction methods, \edit{Atteson} showed that the maximal possible safety radius is $1/2$ \citep{atteson1999performance}.
Single-linkage has a safety radius of exactly $1/2$ \citep{gascuel2004performance}. Single-linkage coincides with $\pi_{M(K_p)}$ up to translation by $\qq$ \citep{ardila2005subdominant}, so it is natural to extend their safety radius result to projections to general Bergman fans. This result follows directly from \cref{lemma: dist to cone boundary}.

\begin{theorem}
    Let $M$ be a matroid on a ground set with $q$ elements. 
    The safety radius of the tropical projection $\pi_M: \TPT{q} \rightarrow \edit{\mathcal{U}_M}$ is $1/2$.
\end{theorem}

\begin{proof}
    \edit{Fix a point $w \in \edit{\uspace_M}$, and let $v \in \TPT{q}$ and $s > 0$. Note that $\dtr(w, v) < 2s\wmin$ if and only if there exists $\epsilon \in (-s \wmin, s\wmin)^q$ such that $w + \epsilon \sim v$.}
    As $\pi_M$ is a function of the equivalence classes of $\TPT{q}$, it therefore suffices to show that $1/2$ is the largest $s$ such that $\pi_M(v)$ is in the relative interior of $K$ for every $v$ such that $\dtr(w,v) < 2 s \wmin$.

    \edit{If $s = 1/2$, then} $v$ satisfies $\dtr(\edit{w, v}) < \wmin$, \edit{and} by \Cref{lem: non-expansive} $\dtr(\edit{w, \pi_M(v)}) < \wmin$. Together with \Cref{lemma: dist to cone boundary}, this implies that $\pi_M(v) \in \relint K$. On the other hand, by \Cref{lemma: dist to cone boundary}, there are points on the boundary $\partial K$, which are exactly distance $\wmin$ from $w$ in the tropical metric. Therefore, we conclude that:
    \[
        1/2 = \sup \{ s: \forall \, \epsilon \in (-s\wmin, s\wmin)^q, \pi_M(w + \epsilon) \in \relint K \}. \qedhere
    \]
\end{proof}

\begin{example}
    \edit{
    We continue with $M=M(G)$, $K_{S_1}\subset K$, and $w \in \relint K_{S_1}$ from \cref{ex:refinement}. 
    Using \cref{defn: safety radius}, we compute $w_{\min}=2$. 
    Suppose $\epsilon=3/4(1,-1,-1,1,1,-1,-1,-1)$ and $w'=w+\epsilon$. Then $\pi_M(w')=1/4(17,17,2,11,11,-3,9,9)$, which is still in $\relint K_{S_1}$. If instead we consider the maximal cones of the cyclic Bergman fan, we find that $\pi_M(w')$ is contained in a different maximal cone than $w$. This is indicated by the flip from $w_4<w_7$ in $w$ to $w_4>w_7$ in $\pi_M(w')$. Thus, $w_{\min}$ depends on the how one elects to refine $\berg(M)$.
    }
\end{example}

\section{Safety Radius for Fermat--Weber Points}
\label{sec:FW_Safety}

In this section, we study the safety radius of \edit{an $M$-ultrametric} estimate.
\edit{An \emph{$M$-ultrametric} estimates takes a sample from $\TPT{q}$ and outputs a single $M$-ultrametric or set of $M$-ultrametrics that summarize the data.}
\rot{We will consider the $M$-ultrametric estimate given by tropical projections of tropical Fermat-Weber points (see \Cref{defn: our method}).}
The tropical projection is necessary because even \edit{for} data in \edit{$\mathcal{U}_M$}, the Fermat-Weber points of the data may lie outside of \edit{$\mathcal{U}_M$} \citep{lin2017convexity, Lin2016TropicalFP}.

\begin{definition}[Fermat-Weber function, Fermat-Weber points]
    \label{def:FW}
    For $S = \{ \sample \} \subset X$, where $X$ is a metric space with distance $d$, the corresponding \emph{Fermat-Weber function} is 
    \[ f_S(x) := \frac{1}{n}\sum_{i=1}^n d(x,v_i). \]
    A \emph{Fermat-Weber point for $S$ with respect to $d$} is any point $x^*$ minimizing the Fermat-Weber function on $X$:
    \begin{equation*}
        \label{eq:FW_point}
        x^* \in \argmin\limits_{x \in X} f_S(x).
    \end{equation*}
    In general, Fermat-Weber points are not unique.
    Thus, we denote the \emph{set of all Fermat-Weber points} of $S$ by $\fw(S) = \fw(\sample)$.
\end{definition}

We note that when $d$ is the tropical metric, $f_S$ is convex (in the classical sense) for any dataset $S ~\edit{\subset \TPT{q}}$.
We use the Fermat-Weber points under the tropical metric to define an $M$-ultrametric estimate.

\begin{definition}
    \label{defn: our method}
    A \emph{Fermat-Weber $M$-ultrametric point} for data $\sample \in \edit{\TPT{q}}$ is any point in $\pi_M(\fw(\sample))$, where $\fw(\sample)$ is taken with respect to the tropical metric.
    The method that computes $\pi_M(\fw(\sample))$ is called the \emph{Fermat-Weber $M$-ultrametric method}.
\end{definition}

\rot{First, we will show that for a dataset \edit{$S\subset\mathcal{U}_M$}, the intersection of $\fw(S)$ with \edit{$\mathcal{U}_M$} is exactly the projection of $\fw(S)$ in $\TPT{q}$ onto \edit{$\mathcal{U}_M$}.}

\begin{theorem}
    \label{thm: symmetric tropical median}
    \edit{For any matroid $M$, and} data $S = \{ \sample \}~\edit{\subset \mathcal{U}_M}$, we have
    \begin{equation}
        \label{eqn: symmetric tropical median}
        \fw(\sample) \cap \edit{\mathcal{U}_M} = \pi_M (\fw(\sample)).
    \end{equation}
    In particular, $\fw(\sample) \cap \edit{\mathcal{U}_M} \neq \emptyset$.
\end{theorem}

\begin{proof}
    Note that every point in $\fw(v_1, \dots, v_n) \cap \edit{\mathcal{U}_M}$ is fixed under the projection $\pi_M$, and hence $\fw(v_1, \dots, v_n) \cap \edit{\mathcal{U}_M} \subseteq \pi_M (\fw(\sample))$.

    Next, we show that $\pi_M(\fw(v_1,\dots, v_n))$ is contained in $\fw(v_1, \dots, v_n) \cap \edit{\mathcal{U}_M}$. 
    For each $i = 1, \ldots, n$:
    \[ \dtr(x, v_i) \geq \dtr(\pi_M(x), \pi_M(v_i)) = \dtr(\pi_M(x), v_i) \text{ for all } x \in \TPT{q}. \]
    The inequality follows from the non-expansivity of $\pi_M$ (\Cref{lem: non-expansive}), and the equality comes from the the fact that $v_i \in \edit{\mathcal{U}_M}$.
    Thus, $f_S(\pi_M(x)) \leq f_S(x)$.
    If $x$ is a Fermat-Weber point of $S$, then $f_S(x)$ is already minimized, so $f_S(\pi_M(x)) = f_S(x) = \min_{y \in \TPT{q}} f_S(y)$.
    This means $\pi_M(x)$ is also a Fermat-Weber point.
    Hence, $\pi_M(\fw(v_1,\dots, v_n)) \subseteq \fw(v_1, \dots, v_n) \cap \edit{\mathcal{U}_M}$.

    The set of Fermat-Weber points is non-empty\edit{.  Therefore, we have $$\pi_M(\fw(\sample)) = \fw(\sample) \cap \edit{\mathcal{U}_M} \neq \emptyset. \qedhere$$}
\end{proof}

This result enables us to compute the Fermat-Weber points in $\TPT{q}$, which is a classically convex optimization problem, and then project to $\edit{\mathcal{U}_M}$, rather than optimizing $f_S$ over the polyhedral fan $\edit{\uspace_M}$. 
\edit{If $S\not \subset \mathcal{U}_M$, it is possible that $\fw(S)\cap \mathcal{U}_M =\emptyset$, as \Cref{ex: S not in UM} demonstrates. Thus, when $S\subset \mathcal{U}_M$ is perturbed by even a small amounts of noise,  $\pi_M\!\left(\fw(S+\epsilon) \right) \subseteq \fw(S + \epsilon)$ may not hold.}
\begin{example}
    \label{ex: S not in UM}
    \edit{Let $M$ be the matroid in our running example and $S=\{v_1,v_2,v_3\}\subset \mathcal{U}_M$ be given by the rows in the left matrix below.
    \begin{equation*}
        S=\begin{pmatrix}
            5 & 5 & 1 & 2 & 2 & 0 & 3 & 3\\
            5 & 5 & 1 & 2 & 2 & 0 & 4 & 4\\
            5 & 5 & 1 & 2 & 2 & 0 & 5 & 5
        \end{pmatrix}, \quad
        S + \epsilon=\begin{pmatrix}
            5 & 5 & 1.01 & 2 & 2.01 & -0.01 & 2.99 & 2.99\\
            5.01 & 4.99 & 1 & 1.99 & 2.01 & 0 & 4 & 4\\
            5.01 & 5 & 0.99 & 2.01 & 2 & 0 & 5 & 4.99
        \end{pmatrix}
    \end{equation*}
    The unique tropical Fermat-Weber point of $S$ is $v_2$ \cite[Lemma~8]{Lin2016TropicalFP}.
    Hence, $\pi_M\!\left(\fw(S) \right)= \{v_2\}$. Let $\epsilon$ be such that $S + \epsilon = \{v^\prime_1,v^\prime_2,v^\prime_3 \}$ yields the matrix to the right where $v^\prime_i$ is the $\ith$ row of the  matrix. One can verify that $\fw(S + \epsilon) = \{v^\prime_2\}$ is also a singleton.  As $v^\prime_2\notin \mathcal{U}_M$, we get $\pi_M\!\left(v^\prime_2 \right) = (4.99,4.99,1,1.99,1.99,0,4,4) \neq v^\prime_2$, which shows $\pi_M\!\left( \fw(S + \epsilon)\right) \not\subset \fw(S+\epsilon)$.}
\end{example}

We now turn to computing the safety radius of the Fermat-Weber points\edit{, combined with a projection to $\uspace_M$}.
For the rest of the section, $S = \{ \sample \} \subset \edit{\TPT{q}}$. For noise vectors $\epsilon_1, \ldots, \epsilon_n \in \R^q$, set $S + \epsilon := \{ v_1 + \epsilon_1, \dots, v_n+\epsilon_n \}$.

\begin{definition}[\edit{Fermat-Weber} $M$-Ultrametric Safety Radius]\label{defn:FW_Multra_SR}
    \rot{Fix a matroid $M$ and a map $\pi_M : \TPT{q} \rightarrow \edit{\mathcal{U}_M}$. The Fermat-Weber $M$-ultrametric \rot{projection} method $S \mapsto \pi_M(FW(S))$ has a safety radius $s \in \R_{\geq 0}$} if, for any finite sample $S \subset \TPT{q}$, maximal cones $K$ of $\nested(M)$, and $\epsilon_1, \ldots, \epsilon_n \in \R^q$, the following holds: 
    \[
        \text{if } \fw(S) \cap \relint  K \neq \emptyset \text{ and } \sum_{j = 1}^n \|\epsilon_j\|_{\tr} < s \cdot \tildewmin \text{, then } \pi_M(\fw (S+\epsilon)) \cap \relint K \neq \emptyset,
    \]
    where $\tildewmin = \max_{w \in \pi_M(\fw(S)) \cap K} \wmin$.
\end{definition}

\edit{In what follows we will use $\tilde{w}$ to denote a vector in $\pi_M(\fw(S)\cap K)$ that achieves the maximum in the definition of $\tilde{w}_{\min}$.}

The Fermat-Weber points of $S$ are the minimizers (also known as an M-estimator\edit{s}) for the \emph{loss function} $f_S(x)$.
For any loss function $f(x)$, \rot{we can compute} the stability of a minimizer \rot{using} a lower bound on the loss function $f(x)$ in terms of the distance from $x$ to the set of minimizers \citep{hayashi2000extremum}. 
The following lemma gives such a bound for the tropical Fermat-Weber function.

\begin{lemma}
    \label{lem:FWfunction_reverse_Lipschitz}
    For any sample $S = \{ \sample \} \subset \TPT{q}$, we have the following bound
    \[
        \frac{1}{n} \dtr (x,\fw(S)) \leq f_S(x) - \min_{y \in \TPT{q}} f_S(y),
    \]
    where $\dtr (x,\fw(S)) = \min_{y \in \fw(S)} \dtr(x, y)$.
\end{lemma}

Before proving \Cref{lem:FWfunction_reverse_Lipschitz}, we present a lemma which we use frequently.
\begin{lemma}\label{lem:perturbed_extrema}
    For $U \subset [q]$, $x \in \R^q \setminus \R \qq$ and all $0 < t < \max^1 x - \max^2 x$:
    \[
        \max_e (x + t\qq_U)_e = \begin{cases}
            t+\max_e x_e &\text{if }U\cap\argmax_e x_e \neq \emptyset, \\
            \max_e x_e &\text{otherwise}. \\
        \end{cases}
    \]
    Similarly, for all $0 < t < \min^2 x - \min^1 x$:
    \[
        \min_e (x + t\qq_U)_e = \begin{cases}
            t+\min_e x_e &\text{if }\argmin_e x_e \subset U, \\
            \min_e x_e &\text{otherwise}. \\
        \end{cases}
    \]
\end{lemma}
\begin{proof}
    We prove the result for $\max_e(x+t\qq_U)_e$; the argument for $\min_e(x+t\qq_U)_e$ is similar.
    Suppose $a \in U \cap \argmax_e x_e$; then 
    \[
        (x+t \qq_U)_a = t+x_a = t + \max_e x_e \geq \max_e( x + t \qq_U)_e.
    \]
    Hence $\max_e(x+t\qq_U)_e = (x+t\qq_U)_a = t+\max_e x_e$.
    Now suppose that $U \cap \argmax_e x_e = \emptyset$. Then for all $a \in \argmax_e x_e$, $(x+ t \qq_U)_a = x_a = \max_e x_e$. For $b \notin \argmax_e x_e$:
    \[
        (x+ t \qq_U)_b \leq t + x_b \leq t + x_a - (\textstyle\max_e^1 x_e - \textstyle\max_e^2 x_e) \leq x_a = (x+ t \qq_U)_a.
    \]
    Hence $\max_e (x+ t \qq_U)_e = (x+ t \qq_U)_a = \max_e x_e$.
\end{proof}
We can now prove \Cref{lem:FWfunction_reverse_Lipschitz}.

\begin{proof}
    To simplify the notation, we set $f(x) = f_S(x)$ throughout the proof.
    Fix $x \in \TPT{q}$, and let $y$ be any point in $\fw(S)$ which minimizes the tropical distance to $x$. The inequality in the lemma holds trivially for all $x \in \fw(S)$, so let us assume $x \notin \fw(S)$, and therefore~$x \neq y$.
    
    Let $\edit{d :}= \dtr(x,y) = \dtr(x,\fw(S)) >0$. 
    Consider the tropical unit vector $(x-y)/\edit{d} \in \TPT{q}$. 
    The derivative of $f(y)$ in the $(x-y)/\edit{d}$ direction is
    \[
        \partial_{(x-y)/\edit{d}}f(y) \coloneqq \lim_{t \rightarrow 0^+} \frac{f(y+\frac{t}{\edit{d}}(x-y)) - f(y)}{t}.
    \]
    We will show that this directional derivative is bounded below by $1/n$:
    \begin{align}\label{eq:direc_deriv_bounded}
        1/n \leq \partial_{(x-y)/\edit{d}} f(y).
    \end{align}
    As $f$ is convex, $t^{-1}(f(y+ \frac{t}{\edit{d}}(x-y))- f(y))$ does not increase as $t \to 0^+$. Thus, \Cref{eq:direc_deriv_bounded} implies that there is some sufficiently small $t$ such that $1/n \leq t^{-1}(f(y+ \frac{t}{\edit{d}}(x-y))- f(y))$.
    For such~$t$, we have
    \begin{align*}
        \frac{1}{n} &\leq \frac{f(y + \frac{t}{\edit{d}} (x - y))-f(y)}{t} = \frac{f((1 - \frac{t}{\edit{d}}) y + \frac{t}{\edit{d}} x) - f(y)}{t} \\
        &\leq \frac{(1 - \frac{t}{\edit{d}})f(y) + \frac{t}{\edit{d}}f(x) - f(y)}{t} = \frac{1}{\edit{d}} (f(x) - f(y)).
    \end{align*}
    The second inequality holds because $f$ is convex.
    Clearing $\edit{d}$ from the denominators gives the inequality in the statement of the lemma.
    It therefore suffices to prove \edit{the inequality in }\Cref{eq:direc_deriv_bounded}.

    To bound the derivative of $f(y)$ in the $(x-y)/\edit{d}$ direction, we first consider directional derivatives in the direction of certain zero-one vectors. 
    The tropical line segment from $y$ to $x$ is a concatenation of classical line segments \citep{develin2004tropicalconvexity}. 
    The number $k$ of classical line segments in the tropical line segment from $y$ to $x$ is one less than the number of distinct entries in $x - y$.  Let $\argmax^j_a (x_a - y_a) \subseteq [q]$ be the coordinates of $x_a - y_a$ achieving the $j$th maximum.  Then,
    the $\ith$ line segment increases exactly the entries indexed by $U_i$, where $U_i := \bigcup_{1 \leq j \leq i} \argmax^j_a (x_a - y_a) \subseteq [q]$. %are the coordinates of $x_a - y_a$ achieving the $j$th maximum.
    \rot{We note that 
    \begin{gather*}
        \emptyset \subset U_1 \subset \dots \subset U_k \subset [q], \\
        \forall i: \;\textstyle\argmin_a(x_a - y_a) \cap U_i = \emptyset, \\
        \forall i: \;\rot{\textstyle\argmax_a(x_a - y_a) \subset U_i}.
    \end{gather*}
    }
    The $\ith$ line segment is parallel to $u_i := \qq_{U_i}$.
    
    We will show that the directional derivatives $\partial_{u_i} f(y)$ are at least $1/n$.
    
    \begin{claim}
        For all $u_i$, the directional derivative of $f$ at $y$ in the direction of $u_i$ satisfies
        $$\frac{1}{n} \leq \partial_{u_i} f(y) \coloneqq \lim_{t \rightarrow 0^+} \frac{f(y+t u_i) - f(y)}{t}.$$
    \end{claim}
    
    \begin{proof}[Proof of Claim 1]
        As $y$ minimizes $f$, we know $\partial_{u_i} f(y) \geq 0$. 
        We will show $\partial_{u_i}f(y)$ is a positive integer multiple of $1/n$.
        
        First we prove that $\partial_{u_i}f(y)$ is non-zero. 
        Suppose for a contradiction that $\partial_{u_i}f(y)$ is zero. 
        Since $f$ is a piecewise linear function with finitely many pieces, for sufficiently small $t$, we have $f(y+t u_i) = f(y)$. 
        Therefore, $y + t u_i$ is a Fermat-Weber point. 
        Moreover, \rot{by \Cref{lem:perturbed_extrema} and the definition of each $U_i$ set, for sufficiently small $t$ we have:}
        \begin{align*}
            \dtr(x, y+t u_i) &= \max_{a} ( x_a - y_a - t u_{ia} ) - \min_{a} ( x_a - y_a - t u_{ia} ) \\
                &= \max_{a} ( x_a - y_a - t ) - \min_{a} ( x_a - y_a  - 0) \\
                &< \max_{a} ( x_a - y_a ) - \min_{a} ( x_a - y_a ) \\
                &= \dtr(x,y).
        \end{align*}
        This contradicts the assumption that $y$ is a closest Fermat-Weber point to $x$, so we conclude that $\partial_{u_i}f(y) \neq 0$.

        Now writing out $\partial_{u_i}f(y)$ explicitly:
        \edit{\begin{align*}
            \partial_{u_i}f(y) &= \lim_{t \rightarrow 0^+} \frac{1}{nt} \sum_{j=1}^n \dtr(y+tu_i,v_j) - \dtr(y,v_j) \\
                % &= \frac{1}{n} \sum_{j=1}^n \lim_{t \rightarrow 0^+} t^{-1} \left( \dtr(y+tu_i,v_j) - \dtr(y,v_j) \right) \\
                &= \frac{1}{n} \sum_{j=1}^n \lim_{t \rightarrow 0^+} t^{-1} \\
                &\times \left[\max_a (y_a + t u_{ia} - v_{ja}) - \min_a (y_a + t u_{ia} - v_{ja}) - \max_a (y_a - v_{ja}) + \min_a (y_a - v_{ja})\right].
        \end{align*}}
        If \rot{$y \nsim v_j$}, then let $t>0$ be less than the smaller of $\max^1_a (y_a - v_a) - \max^2_a (y_a - v_a)$ and $\min^2_a (y_a - v_a) - \min^1_a (y_a - v_a)$.
        If \rot{$y \sim v_j$}, then any $t$ is sufficiently small. Then \rot{using \Cref{lem:perturbed_extrema}} for each $i$, we can expand $\max_a (y_a + t u_{ia} - v_{ja})$:
        \begin{align*}
            \max_a (y_a + t u_{ia} - v_{ja}) &= \begin{cases}
                \max_a(y_a-v_{ja}) + t & U_i \cap \argmax_a(y_a - v_{ja}) \neq \emptyset, \\
                \max_a(y_a-v_{ja}) & U_i \cap \argmax_a(y_a - v_{ja}) = \emptyset. 
            \end{cases}
        \end{align*}
        Similarly:
        \begin{align*}
            \min_a (y_a + t u_{ia} - v_{ja}) &= \begin{cases}
                \min_a(y_a-v_{ja}) + t & \argmin_a (y_a - v_{ja}) \subset U_i, \\
                \min_a(y_a-v_{ja}) & \argmin_a (y_a-v_{ja}) \not \subset U_i.
            \end{cases}
        \end{align*}
        Hence, the \edit{summand} of the directional derivative \edit{is} 
        \edit{\begin{align*}
            \lim_{t \rightarrow 0^+} t^{-1} & ( \dtr(y+tu_i,v_j) - \dtr(y,v_j) ) \\
            &= \lim_{t \rightarrow 0^+} t^{-1} \\
            &\times \left[\max_a (y_a + t u_{ia} - v_{ja}) - \min_a (y_a + t u_{ia} - v_{ja}) - \max_a (y_a - v_{ja}) + \min_a (y_a - v_{ja})\right] \\
            &= \begin{cases}
                1 & U_i \cap \argmax_{a} (y_a-v_{ja}) \neq \emptyset, \argmin_a (y_a-v_{ja}) \not\subset U_i, \\
                -1 & U_i \cap \argmax_{a} (y_a-v_{ja}) = \emptyset, \argmin_a (y_a - v_{ja}) \subset U_i, \\
                0 & \text{otherwise.}
            \end{cases}
        \end{align*}}
        These cases correspond to the cases where the maximum $\max_a (y_a + t u_{ia} - v_{ja})$ is increasing in $t$ but the minimum $\min_a (y_a + t u_{ia} - v_{ja})$ is not, the minimum $\min_a (y_a + t u_{ia} - v_{ja})$ is increasing in $t$ but the maximum $\max_a (y_a + t u_{ia} - v_{ja})$ is not, and the finally the case where they are either both increasing or both fixed in $t$. 
        Hence the directional derivative $\partial_{u_i}f(y)$ is an integer multiple of $1/n$. We conclude that $\partial_{u_i}f(y) \geq 1/n$.
    \end{proof}

    Next, we use the bounds on the directional derivatives $\partial_{u_i} f(y)$ to obtain the desired bound on $\partial_{(x-y)/\edit{d}} f(y)$.
    We can write $x - y$ as a classical linear combination of $u_1, \ldots, u_k$:
    \[x - y = \R \qq \odot \sum_{i=1}^k \gamma_i u_i, \]
    where $\gamma_i = \max^i_a (x_a - y_a) - \max^{i+1}_a (x_a - y_a)$.
    Moreover, $\gamma_i > 0$ for all $i$, and $\sum_{i=1}^k \gamma_i = \dtr(x, y) = \edit{d}$.
    This allows us to write $(x-y)/\edit{d}$ as a convex combination of the $u_i$:
    \begin{equation}
        \label{eqn: convex combo}
        \frac{x - y}{\edit{d}} = \R \qq \odot \sum_{i=1}^k \lambda_i u_i, 
    \end{equation}
    where $\lambda_i = \gamma_i/\edit{d}$ for $i = 1, \ldots, k$; note that $\lambda_i > 0$, and $\sum_{i=1}^k \lambda_i = 1$.
    
    \begin{claim}
        Let $\lambda_1, \ldots, \lambda_k \geq 0$, and $\sum_{i=1}^k \lambda_i = 1$.
        For sufficiently small $t$, $f(y + t \sum_i \lambda_i u_i) = \sum_{i} \lambda_i f(y + t u_i)$.
    \end{claim}
    \begin{proof}[Proof of Claim 2]
        It suffices to consider a single term of the Fermat--Weber function, so we consider the distance to a single data point $v \in S$. 

        We must pick $t$ small enough so that $\argmax_a (y_a-v_a + t\textstyle\sum_i\lambda_i u_{ia}) \subseteq \argmax_a (y_a - v_a)$.
        We note that if $y = v$, then this holds for any $t$. Otherwise, as before, we require that $t < \max^1_a (y_a - v_{a}) - \max^2_a (y_a - v_{a}), \min^2_a (y_a - v_{a}) - \min^1_a (y_a - v_{a})$.
        The distance to $v$ is
        \begin{align*}
            \dtr(y+t\textstyle\sum_i\lambda_i u_i, v) = \displaystyle\max_a (y_a-v_a + t\textstyle\sum_i\lambda_i u_{ia}) - \displaystyle\min_a (y_a-v_a + t\textstyle\sum_i\lambda_i u_{ia}).
        \end{align*}
        We picked $t$ small enough that any maximal coordinate for $y_a-v_a + t\textstyle\sum_i\lambda_i u_{ia}$ must have also been maximal for $y_a-v_a$. Similarly, any minimal coordinate for $y_a-v_a + t\textstyle\sum_i\lambda_i u_{ia}$ must also be minimal for $y_a-v_a$. This allows us to explicitly expand the maximum and minimum terms.
        
        We first consider the maximum term in this distance. Consider $\argmax_a(y_a-v_a) \cap U_i$ for $i = 1, \ldots, k$. If $\argmax_a(y_a-v_a) \cap U_i = \emptyset$ for all $i$, then by our choice of $t$, a perturbation by $t\textstyle\sum_i\lambda_i u_{ia}$ leaves the maximum unchanged.
        Furthermore, since $\lambda_1 + \cdots + \lambda_k = 1$, we have
        \begin{align*}
            \max_a (y_a-v_a + t\textstyle\sum_i\lambda_i u_{ia}) = \displaystyle\max_a (y_a-v_a)
                    &= \textstyle\sum_i \lambda_i \displaystyle\max_a (y_a-v_a + tu_{ia}).
        \end{align*}
        Otherwise, let $r$ be the smallest index such that $\argmax_a(y_a-v_a) \cap U_r \neq \emptyset$. Then,
        \begin{align*}
            \max_a (y_a-v_a + t\textstyle\sum_i\lambda_i u_{ia}) &= \max_a (y_a-v_a) + \sum_{i = r}^k \lambda_i t \\
                    &= \sum_{i=1}^{r-1} \lambda_i \max_a (y_a-v_a) + \sum_{i = r}^k \lambda_i (\max_a (y_a-v_a) + t) \\
                    &= \sum_{i=1}^k \lambda_i \max_a (y_a-v_a + tu_{ia}).
        \end{align*}
        We perform an analogous computation for the minimum term. 
        Consider $\argmin_a (y_a-v_a) \setminus U_i$ \edit{for} $i = 1, \ldots, k$. 
        If $\argmin_a (y_a-v_a) \setminus U_i$ is non-empty for all $i$, then the minimum remains unchanged and we can write:
        \begin{align*}
            \min_a (y_a-v_a + t\textstyle\sum_i\lambda_i u_{ia}) = \displaystyle\min_a (y_a-v_a)
                    = \sum_i \lambda_{i=1}^k \displaystyle\min_a (y_a-v_a+tu_{ia}).
        \end{align*}
        Otherwise, let $s$ be the smallest $i$ such that $\argmin_a (y_a-v_a) \setminus U_i = \emptyset$. Then,
        \begin{align*}
            \min_a (y_a-v_a + t\textstyle\sum_{i=1}^k \lambda_i u_{ia}) &= \min_a (y_a-v_a) + \sum_{i = s}^k \lambda_i t \\
                    &= \sum_{i=1}^{s-1} \lambda_i \min_a (y_a-v_a) + \sum_{i = s}^k \lambda_i (\min_a (y_a-v_a) + t) \\
                    &= \sum_{i=1}^k \lambda_i \min_a (y_a-v_a + tu_{ia}).
        \end{align*}
        So the claim holds for a single datapoint $v$, and by summing over our sample we get the desired result.
    \end{proof}
    
    We can now compute the directional derivative $\partial_{(x-y)/\edit{d}}f(y)$.
    Recall \Cref{eqn: convex combo}, which gives $(x - y)/\edit{d} \odot \R \qq$ as a convex combination of the $u_i$.
    We have
    \begin{align*}
        \partial_{(x-y)/\edit{d}}f(y)
            &= \lim_{t \rightarrow 0^+} \frac{f(y+t\sum_{i=1}^k \lambda_i u_i)- f(y)}{t} \\
            &= \lim_{t \rightarrow 0^+} \frac{ \big(\sum_{i=1}^k \lambda_i f(y+t u_i)\big) - (\sum_{i=1}^k \lambda_i) \cdot f(y)}{t} \\
            &= \sum_{i=1}^k \lambda_i \lim_{t \rightarrow 0^+} \frac{ f(y+t u_i)- f(y)}{t} = \sum_{i=1}^k \lambda_i \partial_{u_i} f(y).
    \end{align*}

    Since $\partial_{u_i} f(y) \geq 1/n$ for all $i$, this implies $\partial_{(x-y)/\edit{d}}f(y) \geq 1/n$.
\end{proof}

\begin{lemma}
    \label{lem: hausdorff}
    Let $S = \{ \sample \}$ be a sample in $\TPT{q}$, and let $S+\epsilon = \{ v_1+\epsilon_1, \dots, v_n+\epsilon_n \}$ for some $\epsilon_1, \ldots, \epsilon_n \in \R^q$. Then,
    \[
        d_{H}(\fw(S), \fw(S+\epsilon)) \leq 2 \cdot \sum_{j = 1}^n \|\epsilon_j\|_{\tr},
    \]
    where $d_H$ is the Hausdorff distance using $\dtr$, which is given by:
    \[
        d_H(A,B) = \max \{ \sup_{x \in A} \dtr(x,B), \sup_{y \in B} \dtr(A, y) \}.
     \]   
\end{lemma}

\begin{proof}
    We denote the Fermat-Weber functionals for the sample $S$ and $S+\epsilon$ by $f_{S}$ and $f_{S+\epsilon}$ respectively. 
    Let $y$ be any Fermat-Weber point of $S+\epsilon$. 
    By \Cref{lem:FWfunction_reverse_Lipschitz}, for all $z$ in $\fw(S)$:
    \begin{align*}
        \frac{1}{n} \dtr(\fw(S), y) &\leq f_S(y) - f_S(z) \\
                &\leq f_S(y) - f_S(z) + f_{S+\epsilon}(z) - f_{S+\epsilon}(y) \\
                &= \frac{1}{n} \sum_{j=1}^n \dtr(y,v_j) - \dtr(z,v_j) + \dtr(z,v_j+\epsilon_j) - \dtr(y,v_j+\epsilon_j).
    \end{align*}
    The triangle inequality implies $\dtr(y, v_j) - \dtr(y, v_j + \epsilon_j) \leq \dtr(v_j, v_j + \epsilon_j)$, and also $ \dtr(z, v_j + \epsilon_j) - \dtr(z, v_j) \leq \dtr(v_j, v_j + \epsilon_j)$.
    Thus,
    
    \begin{align*}
        \frac{1}{n} \dtr(\fw(S), y)
                &\leq \frac1n \sum_{j=1}^n 2 \, \dtr(v_j, v_j + \epsilon_j) \\
                &= \frac{1}{n} \sum_{j=1}^n 2 \, \|\epsilon_j \|_{\tr}.
    \end{align*}

    Hence:
    \begin{align*}
        \sup_{y \in \fw(S+\epsilon)} \dtr(\fw(S), y) \leq 2 \sum_{j=1}^n \|\epsilon_j \|_{\tr}.
    \end{align*}
     A symmetrical argument proves that $\sup_{z \in \fw(S)} \dtr(z, \fw(S + \epsilon)) \leq 2 \sum_{j} \|\epsilon_j \|_{\tr}$ as well.
\end{proof}

This bound on $d_{H}(\fw(S), \fw(S+\epsilon))$ is not necessarily optimal for all $n$. 
For example, when $n=1$ the factor of two is superfluous. 
For a view of the empirical tightness of this bound, see \Cref{fig:haust_dist_histograms}.

Having identified concrete bounds on the stability of the tropical Fermat--Weber points, we can now state our main theorem.

\begin{theorem}
    \label{thm: tropical safety radius}
    For any matroid $M$, the Fermat-Weber $M$-ultrametric safety radius \edit{is} 1/2.
\end{theorem}

\begin{proof}
    Let $S = \{ \sample \}$ be a sample which has a tropical Fermat-Weber point in the relative interior of the cone $K \in \edit{\nested(M)}$, and let $w$ be the Fermat-Weber point in $K$ that is furthest from the boundary of $K$, such that $\tildewmin = \dtr(w,\partial K)$. Suppose the sum of tropical norms of $\epsilon_1, \dots, \epsilon_n$ is at most $\tildewmin/2$. By \Cref{lem: hausdorff}:
    \[
        d_H(\fw(S), \fw(S+\epsilon)) \leq \tildewmin.
    \]
    Therefore $S+\epsilon$ has some Fermat-Weber point $u$ within $\tildewmin$ of $w$, which ensures that \edit{$\pi_M(u)$} is also in $\relint K$.
\end{proof}

\edit{
\section{Stochastic Safety Radius and Experimental Results} 
\label{sec:stochastic safety radius}
}

\edit{
The safety radius, as introduced by Atteson in \citep{atteson1999performance}, provides conditions for when a distance-based method returns a correct topology after introducing Gaussian noise to a tree metric. This definition of a safety radius is {\em deterministic} even though
the input data are often multivariate random variables. 
This motivated Gascuel and Steel \citep{gascuel2016stochastic} to define a ``stochastic safety radius'' in terms of a probability distribution: assuming the probability of returning the correct tree topology is at least $1-\eta$, how much variation can our error have?}

\edit{
In this section, we extend the latter definition to Fermat--Weber $M$-ultrametric reconstructions for general \edit{matroids}, showing that for any matroid and desired probability of success $1-\eta$, there is some non-zero stochastic safety radius. In fact, for sufficiently large $n,q$, $\fw$ $M$-ultrametric stochastic safety radii can be taken arbitrarily close to $1/8$.
}

\begin{definition}[FW $M$-Ultrametric Stochastic Safety Radius]
    \label{defn: stochastic safety radius}
    Fix \edit{a sample size $n$, a ground set $[q]$, and} a matroid $M$ \edit{on $[q]$}.
    \edit{Some} $s \in \R_{\geq 0}$ is a $\eta$-stochastic safety radius of the Fermat-Weber $M$-ultrametric method if, for any finite sample $S \subset \TPT{q}$, maximal cones $K$ of $\nested(M)$, and i.i.d. random variables $\epsilon_1, \ldots, \epsilon_n \sim N(0,\sigma^2\mathbb{I}_q
    )$, the following holds: 
    \[
        \fw(S) \cap \relint  K \neq \emptyset \text{ and } \sigma < s \cdot \frac{\sqrt{2} \tildewmin}{n\sqrt{\log q}} \, \Rightarrow \,  \mathbb P(\pi_M(\fw (S+\epsilon)) \cap \relint  K \neq \emptyset ) \geq 1 - \eta,
    \]
    where $\tildewmin = \max_{w \in \pi_M(\fw(S)) \cap K} \wmin$.
\end{definition}

We note that in the standard phylogenetic stochastic safety radius, the coefficient of $s$ depends on the number of leaves, $p$, instead of the dimension of the ambient space, $q = \binom{p}{2}$. 
However, for large leaf count $p$, $\sqrt{\log (q)} \approx \sqrt{2\log(p)}$. 
This motivates the factor of $\sqrt{2}$ in our definition.
Our definition is designed to standardize the order of magnitude of $s$ for large $n,q$ and to allow for a meaningful comparison between methods at different scales. This is reflected in the \edit{next} theorem.

\begin{theorem}
    \label{thm: stochastic safety radius}
    For any $\eta > 0$, $n,q \in \mathbb{N}$, and matroid $M$ \edit{on $[q]$}, there is a non-zero \rot{$\fw$ $M$-ultrametric stochastic} safety radius $s > 0$. Moreover, \rot{the stochastic safety radius $s$ can be taken arbitrarily close to $1/8$ for sufficiently large $n,q$}.
\end{theorem}
\begin{proof}
    We will use the notation from \Cref{defn: stochastic safety radius}.
    In particular, $\epsilon_1, \ldots, \epsilon_n \sim N(0, \sigma^2 \mathbb{I}_q)$.
    By \cref{thm: tropical safety radius}, we have $\pi_M(\fw (S+\epsilon)) \cap \relint K$ \edit{is non-empty} whenever $\fw(S)\cap \relint  K\neq \emptyset$ and $\sum_{j=1}^n \|\epsilon_j \|_{\tr} \leq \tildewmin/2$. Let $Z_1, \ldots, Z_n \sim N(0,\edit{\mathbb{I}}_q)$ be standard i.i.d. multivariate normal variables, so that $\epsilon_j \sim \sigma Z_j$, for $j = 1, \ldots, n$. Define $E(q) = \mathbb E \|Z_j\|_{\tr}$ and $V(q) = \text{Var} \|Z_j\|_{\tr}$. Using Chebyshev's inequality:
    \begin{align*}
        \mathbb P \left(\pi_M(\fw (S+\epsilon)) \cap \relint  K \neq \emptyset\right) &\geq \mathbb P\left( \sum_{j=1}^n \|\epsilon_j \|_{\tr} \leq \frac{\tildewmin}{2} \right) \\
            &= \mathbb P \left( \frac{1}{n}\sum_{j=1}^n \|Z_j \|_{\tr} \leq \frac{\tildewmin}{\edit{2n\sigma}} \right) \\
            &= 1 - \mathbb P \left( \frac{1}{n}\sum_{j=1}^n \|Z_j \|_{\tr} - E(q) > \frac{\tildewmin}{\edit{2n\sigma}} - E(q) \right) \\
            &\geq 1 - \mathbb P \left( \left| \frac{1}{n}\sum_{j=1}^n \|Z_j \|_{\tr} - E(q) \right| > \frac{\tildewmin}{\edit{2n\sigma}} - E(q) \right) \\
            &\geq 1 - \frac{\text{Var}(\frac{1}{n} \sum_{j=1}^n \| Z_j\|_{\tr})}{\left(\tildewmin/\edit{(}2n\sigma\edit{)} - E(q)\right)^2} \\
            &= 1 - \frac{V(q)/n}{\left(\tildewmin/\edit{(}2n\sigma\edit{)} - E(q)\right)^2}.
%            &= 1 - \frac{V(q)}{\tildewmin/2\sigma - nE(q)}.
    \end{align*}
    Therefore, \edit{we know that} $\mathbb P \left(\pi_M(\fw (S+\epsilon)) \cap \relint K \neq \emptyset\right) \geq 1-\eta$ whenever \edit{the inequality} \edit{$\eta \geq V(q)n^{-1} (\tildewmin / (2n \sigma) - E(q))^{-2}$ holds\edit{; this inequality holds} for all $\sigma$ such that}
    \begin{align*}
        \sigma \leq \frac{\tildewmin}{2 (n E(q)+\sqrt{nV(q)/ \eta})}.
    \end{align*}
    It therefore suffices to find $s$ such that 
    \[
        s < \frac{\sqrt{ \log q}}{2\sqrt{2} (E(q)+\sqrt{V(q)/n\eta})}.
    \]
    $\edit{E(q),V(q)>0}$ all $q$, \edit{thus}, there is some positive safety radius $s>0$ for all $n,q$.

    It is well known that for i.i.d. normal variables $Z_{j1}, \dots, Z_{jq}$, we have $\mathbb E \edit{[}\max_{l} Z_{jl}\edit{]} / \sqrt{\log q} \rightarrow \sqrt{2}$ as $q \rightarrow \infty$. Therefore, $ E(q) / \sqrt{\log q} = 2 \mathbb E\edit{[} \max_{l} Z_{jl}\edit{]} / \sqrt{\log q} \rightarrow 2 \sqrt{2}$  as $q \rightarrow \infty$. Thus, for all $\eta, \delta > 0$, and sufficiently large $n,q$, we have $\eta$-stochastic safety radius $s> 1/8 - \delta$.
\end{proof}

This result mirrors Proposition 2 of \citep{gascuel2016stochastic}, providing some theoretical intuition for the order of magnitude we might expect for $s$. 
\Cref{thm: stochastic safety radius} provides a benchmark against which to compare safety radius estimates, but it is known to be hard to compute exact tight safety radii values in general. Rather, stochastic stability is best measured through computational simulation\edit{, which we now demonstrate}.

Let $M$ be the graphical matroid from our running example. 
We select a point $v_0$ from the relative interior of a randomly selected maximal cone $K\in\nested(M)$ such that $\|v_0 \|_{\tr}=1$.
Let $S=\{v_j\mid j\in[n]\}\subset\R^q$ be a sample of size $n$ centered on $v_0$. Here we sample via a multivariate Gaussian $v_j\sim\mathcal{N}(v_0,\mathbb{I}_q\sigma)$ with $\sigma=3$ and $n=25$, followed by a projection of the points to $\uspace_M$. Since we require $\{\fw(S)\cap \relint K\} \neq \emptyset$ for some $K\in\nested(M)$, we reject any sample $S$ failing to do so. This explains the decision to project samples to $\uspace_M$. Projection ensures (via \Cref{thm: symmetric tropical median}) that $S\cap\uspace_M\neq\emptyset$, which greatly improves the likelihood that $S\cap\relint K\neq \emptyset$ for a maximal cone $K\in\nested(M)$. Thus, while projection here is not strictly necessary, we find that in practice it helps to reduce sampling rejection rates. We repeat this procedure as necessary, selecting a new $v_0$ for each sample, until we attain a collection of 1,000 such samples.

\edit{To any sample $S$ with $\fw(S)\cap \relint K\neq\emptyset$ for a maximal cone $K\in \nested(M)$, we can apply noise of varying magnitude, and check whether $\{\pi_M\!\left(\fw(S + \epsilon) \right)\cap \relint K\} \neq \emptyset$. This general procedure is outlined in \cref{alg:test_fw}.}

\begin{algorithm}
    \caption{Determine if $\{\pi_M\!\left(\fw(S+\epsilon)\right) \cap \relint K\} \neq\emptyset$.}\label{alg:test_fw}
    \begin{algorithmic}
        \State \textbf{Input}: A sample $S=\{\sample \} \subset \TPT{q}$ such that $W:=\{\fw(S)\cap\relint K\}$ is non-empty for some maximal cone $K\in \nested(M)$. Noise variance $\sigma_\epsilon^2$.
        \State \textbf{Output}: Total noise $\sum_j \|\epsilon_j \|_{\tr}$, $\tilde{w}_{\min}$ and (if possible) a point $x\in \fw(S + \epsilon)$ satisfying $\pi_M(x)\in \relint K$.
        \begin{algorithmic}[1]
            \State Compute $\fw(S)$.
            \State Compute $\tilde{w}_{\min}:=\max_{w\in W}(w_{\min})$ and $\tilde{w}:=\argmax_{w\in W}(w_{\min})$.
            \For {$j=1,\ldots,n$} generate noise $\epsilon_j\sim \mathcal{N}(0,(\tilde{w}_{\min}\sigma_\epsilon)^2)$.
            \EndFor
            \State Compute $P:=\fw(\{v_1+\epsilon_1,\ldots,v_n+\epsilon_n \})$.
            \If {$\exists x \in P$ such that $\pi_M(x) \in \relint K$} $x^* \leftarrow x$ \Else {$~x^*\leftarrow\emptyset$}
            \EndIf
        \end{algorithmic}
        \Return $\tilde{w}_{\min}$, $\sum_j\|\epsilon_j\|_{\tr}$, and $x^*$.
    \end{algorithmic}
\end{algorithm}

The non-trivial portions of \cref{alg:test_fw} are Lines 1, 2, and 5. In Line 1 (and similarly in Line 4) we compute the Fermat--Weber set's hyperplane representation using the linear programming outlined in \citep{sabol2024polytropes}. To compute $\tilde{w}$ in Line 2 we solve for the Chebyshev center of $K$ (under the tropical norm) subject to additional feasibility constraints imposed by $\fw(S)$. A linear program for this problem is feasible by the assumption $\fw(S) \cap K \neq \emptyset$ and bounded by $\fw(S)$. Its solution directly yields $\tilde{w}$ and $\tilde{w}_{\min}$ as optimal decision variables with $\tilde{w}_{\min}>0$ if and only if $\fw(S)\cap\relint K\neq \emptyset$. In fact, this is precisely the computation we perform when deciding whether to accept or reject a sample generated in the manner described in the previous paragraph. In other words, we compute Lines 1-2 when generating our samples.

\edit{After applying noise in Line 3, the points of the perturbed sample $S+\epsilon$ are not $M$-ultrametrics, even if $S\subset \uspace_M$ originally. 
Thus, we cannot leverage \cref{thm: symmetric tropical median} in evaluating Line 5. 
Instead, we compute the preimage of $K$ under $\pi_M$ and check if $\{\fw(S + \epsilon) \,\cap\,\pi_M^{-1}(K)\}$ is nonempty.
The preimage of $K$ is a union of polyhedral cones, each of which arises by resolving disjunctive inequalities associated to the set of partial orders encoded in the nested set structure of $K$. For example, from \cref{fig:M(G)} we see that the nested set structure for $K_{S_1}$ requires $\min\{x_1,x_2\}\geq\min\{x_4,x_5\},\min\{x_7,x_8\}$. Thus, $\pi_M^{-1}(K_{S_1})$ is the union of the four polyhedral cones whose associated half-spaces differ according to the four possible ways to assign minima on the right-hand side of the preceding inequality.
By enumerating these preimage cones, we can perform a sequence of constrained Chebyshev center linear programs over each of them in turn, stopping if we encounter a feasible solution with positive objective function value. Such a solution $x$ provides $x^*=\pi_M(x)$ in Line 5. If no such $x$ is found after checking all cones in the preimage set, we conclude 
$\{\pi_M\!\left(\fw(S+\epsilon)\right)\cap\relint K\}$ is empty.}

\edit{The next two examples demonstrate the empirical tightness of the bound given in \cref{lem: hausdorff} and the Fermat--Weber $M$-ultrametric stochastic safety radius. Both examples use the matroid $M$ from our running example and 1,000 samples of size $n=25$ that are generated in the manner described earlier.}

\begin{example}\label{ex:haussdorff}
    \edit{
    To each of 1,000 samples $S$, we apply uniform noise $\epsilon$ scaled so that $\sum_j\|\epsilon_j\|_{\mathrm{tr}}\approx\tilde{w}_{\min}/2$. That is, we apply noise that meets the threshold stipulated by our bound in \cref{thm: tropical safety radius}. Next, we compute the Fermat-Weber polytrope $P=\fw(S+\epsilon)$, project $\tilde{w}$ onto $P$, and then measure $\dtr(\tilde{w}, \pi_P(\tilde{w}))$. The tropical vertices of $P$ used for the projection can be computed by solving a network flow problem (see \citep{sabol2024polytropes}). \cref{fig:Safety_Rad_Demo} (left) plots these 1,000 distances by the value of $\tilde{w}_{\min}$ with the red dotted line indicating the bound given by \cref{lem: hausdorff}, offering some insight into its empirical tightness.}
\end{example}

\begin{example}\label{ex:SSR}
    \edit{Using the same 1,000 samples as in \cref{ex:haussdorff}, we execute \cref{alg:test_fw} once for each noise level $\sigma_\epsilon \in \{1/32, 1/16, 1/8, 1/4, 1/2, 1, 2, 3,4,5 \}$. 
    % Thus, for each sample, we perform 10 iterations of \cref{alg:test_fw}, one for each level of $\sigma_\epsilon$. 
    At each iteration we record $\sum_{j=1}^n \|\epsilon_j \|_{\tr}$ and whether $\pi_M\!\left(\fw(S + \epsilon) \right)\cap \relint K \neq \emptyset$. The results of the 10,000 trials are depicted in \cref{fig:Safety_Rad_Demo} (right) where the vertical dashed line indicates the noise threshold ratio provided in \cref{thm: tropical safety radius}. It is evident in this demonstration that the bound is extremely conservative, which is line with the observations made in \citep{gascuel2016stochastic} for phylogenetic trees.}
\end{example}

\begin{figure}
    \centering
    \includegraphics[width=1\linewidth]{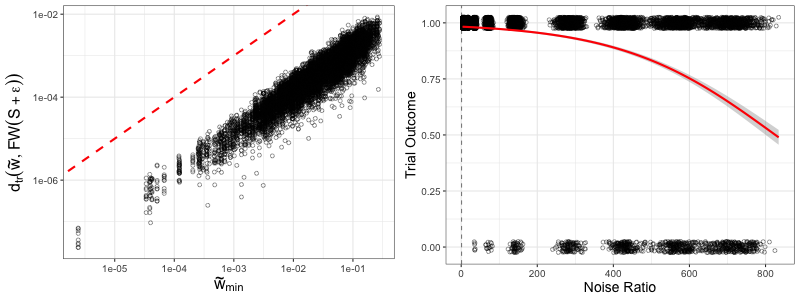}
    \caption{
    \edit{
    Left: Tropical distances between $\tilde{w}$ and $\fw(S + \epsilon)$ after applying noise at the threshold given by \cref{thm: tropical safety radius}. Using 1,000 separate samples, we apply uniform noise proportionally scaled to achieve $\sum_j\|\epsilon_j\|_{\mathrm{tr}}\approx \tilde{w}_{\min}/2$. We write approximately only to account for floating point precision in the computations. The red dashed line indicates the upper-bound on the Hausdorff distance given in \cref{lem: hausdorff}. We plot the results in log-log coordinates for better readability. \\
    Right: Results for 10,000 applications of \cref{alg:test_fw} using the same samples as in \cref{ex:haussdorff}.
    An outcome of $0$ indicates $\pi_M\!\left(\fw(S+\epsilon)\right)\cap \relint K = \emptyset$. 
    To each of the 1,000 separate samples we apply Gaussian noise according to 10 different progressively increasing variances so that we have 10 tests per sample. We plot the x-axis using the ratio of $\sum_j\|\epsilon_j\|_{\mathrm{tr}}/(\tilde{w}_{\min}/2)$ to help normalize across samples. The vertical dashed line at $x=1$ corresponds to the threshold of this ratio required in \cref{thm: tropical safety radius}. A logistic regression (red curve) with standard errors (gray bands) is added to show the overarching trend.}
    }
    \label{fig:Safety_Rad_Demo}
\end{figure}

\edit{In order to determine how closely the results of \cref{ex:SSR} match with other matroids on the same ground set, we repeat \cref{ex:SSR} for all loopless, connected matroids on 8 elements with rank $r(M)\geq 2$. For this we use the matroid database in \textsc{SageMath} \citep{sagemath} to enumerate all matroids on $[8]$. After filtering out those we wish to exclude, we are left with 1,151 matroids. For expediency, we reduce the number of samples per matroid from 1,000 down to 10 while keeping the sample size $n=25$ the same. Additionally, we apply noise only at $\sigma_\epsilon\in\{1/32, 1/8, 1, 3, 5 \}$. The resulting 57,550 trials are depicted in \cref{fig:all_matroids} where we plot the $x$-axis on a log scale to better highlight the outcomes for small $\sigma$.}

\begin{figure}
    \centering
    \includegraphics[width=0.75\linewidth]{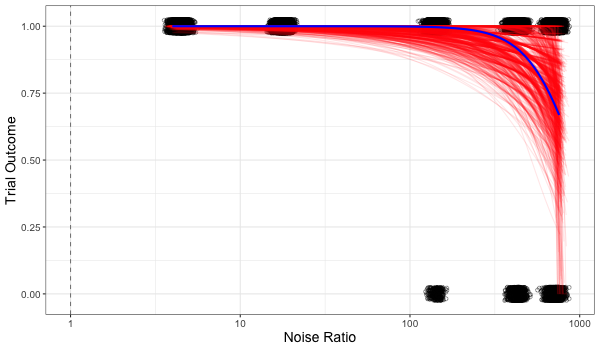}
    \caption{\edit{We repeat \cref{ex:SSR} for all loopless, connected matroids on 8 elements with rank $r(M)\geq 2$. For each matroid we construct 10 samples and vary $\sigma_\epsilon\in \{1/32,1/8,1,3,5 \}$. Each red curve is a logistic regression for the experimental outcomes of a different matroid. The blue curve corresponds to the matroid from our running example. We plot the $x$-axis on a log scale to better depict results $\sigma$ closer to the noise threshold (vertical dashed line at $x=1$).}}
    \label{fig:all_matroids}
\end{figure}

\section{Conclusions}
\edit{In this paper we explore the stability of projection to the space of $M$-ultrametrics for Fermat-Weber sets under additive noise. We extend several known results for the specific case of phylogenetic trees to the more general setting of arbitrary matroids. Our definitions (and proofs) utilize $w_{\min}$ based on the nested set fan under a minimal building set, but the concepts apply equally to any refinement of the Bergman fan with an appropriate update to the formula given in \cref{defn: safety radius}. Our empirical results suggest that the FW M-ultrametric safety radius is very conservative, which was expected given that the ultrametric safety radius is known to be extremely conservative in the case of phylogentic trees.  In future work, it would be interesting to find a tighter bound.}

\backmatter

\bmhead{Acknowledgments}

The authors thank to Yue Ren and Connor Simpson for useful conversations, and the anonymous reviewers for their comments.
RY and JS are partially supported by NSF Statistics Program DMS 2409819. RT receives partial funding from a Technical University of Munich--Imperial College London Joint Academy of Doctoral Studies (JADS) award (2021 cohort, PIs Drton/Monod). SC, RT and RY are grateful to Jane Coons for bringing us together at the University of Oxford during the Workshop for Women in Algebraic Statistics in July 2024 supported by St John's College, Oxford, the L'Oreal-UNESCO For Women in Science UK and Ireland Rising Talent Award in Mathematics and Computer Science (awarded to Jane Coons), the Heilbronn Institute for Mathematical Research, and the UKRI/EPSRC Additional Funding Programme for the Mathematical Sciences.

\begin{appendices}

\section{Additional Computational Results}\label{app:results}

To \edit{study} the empirical tightness of the bound on $d_H\left(\mathrm{FW}(S), \mathrm{FW}(S+\epsilon)\right)$ \edit{given by} \Cref{lem: hausdorff}, we \edit{generate a random sample $S=\{\sample \}$, where each $v_j$ drawn uniformly from the $q$-dimensional unit cube. 
To each $v_j$ we add standard Gaussian noise $\epsilon_j\sim \mathcal{N}(0,\mathbb{I}_q)$, where $\mathbb{I}$ is the $q\times q$ identify matrix. 
We compute $P=\fw(S)$ and $P^\prime=\fw(S+\epsilon)$, and for each vertex $p\in P$ we compute $\dtr(p, \pi_{P^\prime}(p))$. 
We perform similar computations for $\dtr(p^\prime, \pi_P(p^\prime))$ and then take the maximum over all such distances to attain $d_H(\fw(S),\fw(S+\epsilon))$. We repeat this procedure 1,000 times for multiple values of $n,q$, and for each iteration, compute} the \emph{scaled shift}: 
\[
    \edit{d_H(\fw(S), \fw(S+\epsilon)) / (\textstyle  \sum_{j} \|\epsilon_j \|_{\tr}).}
\]
\edit{Results are depicted in \Cref{fig:haust_dist_histograms}. We note that in all trials, the scaled shift ($x$-axis) never exceeds a value of one, and the highest values are attained when $n=2$. We also observe that scaled shift values generally decrease as a sample size increases. This suggests that it may be possible to modify the factor of two included in \cref{lem: hausdorff} to incorporate information on a sample size so as to achieve tighter bounds.}

\begin{figure}[H]
    \centering
    \includegraphics[width = \textwidth]{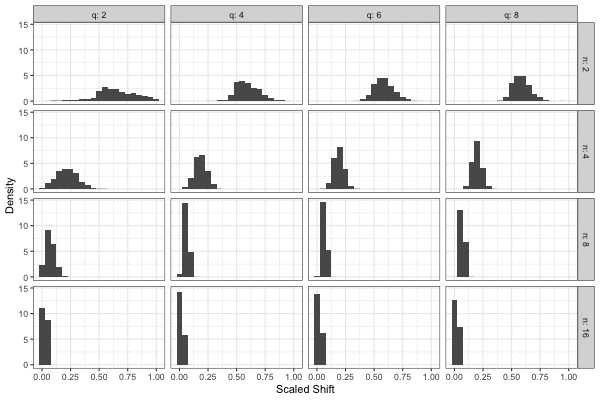}
    \caption{
    Each histogram records the scaled shifts for 1,000 $q$-dimensional samples of $n$ points, for $q = 2,4,6,8$ and $n = 2, 4, 8, 16$. The 1,000 independent samples and perturbations are generated according to the procedure described in Appendix \ref{app:results}.
    }
    \label{fig:haust_dist_histograms}
\end{figure}

\end{appendices}

\bibliography{ref}

\begin{thebibliography}{20}
\providecommand{\natexlab}[1]{#1}
\providecommand{\url}[1]{{#1}}
\providecommand{\urlprefix}{URL }
\providecommand{\doi}[1]{\url{https://doi.org/#1}}
\providecommand{\eprint}[2][]{\url{#2}}
 \bibcommenthead

\bibitem[{Ardila(2004)}]{ardila2005subdominant}
Ardila F (2004) Subdominant matroid ultrametrics. Ann Comb 8(4):379--389.
  \doi{10.1007/s00026-004-0227-1}

\bibitem[{Ardila and Klivans(2006)}]{arilda2006bergmancomplex}
Ardila F, Klivans CJ (2006) The {B}ergman complex of a matroid and phylogenetic
  trees. J Combin Theory Ser B 96(1):38--49. \doi{10.1016/j.jctb.2005.06.004}

\bibitem[{Atteson(1999)}]{atteson1999performance}
Atteson K (1999) The performance of neighbor-joining methods of phylogenetic
  reconstruction. Algorithmica 25(2-3):251--278. \doi{10.1007/PL00008277}

\bibitem[{Bernstein(2020)}]{bernstein2020infinity}
Bernstein DI (2020) {$L$}-infinity optimization to {B}ergman fans of matroids
  with an application to phylogenetics. SIAM J Discrete Math 34(1):701--720.
  \doi{10.1137/18M1218741}

\bibitem[{Buneman(1974)}]{buneman1974properties}
Buneman P (1974) A note on the metric properties of trees. J Combinatorial
  Theory Ser B 17(1):48--50. \doi{10.1016/0095-8956(74)90047-1}

\bibitem[{Develin and Sturmfels(2004)}]{develin2004tropicalconvexity}
Develin M, Sturmfels B (2004) Tropical convexity. Doc Math 9:1--27.
  \doi{10.4171/dm/154}

\bibitem[{Dress et~al.(2005)Dress, Holland, Huber, Koolen, Moulton, and
  Weyer-Menkhoff}]{dress2005delta}
Dress A, Holland B, Huber KT, et~al (2005) {$\Delta$} additive and {$\Delta$}
  ultra-additive maps, {G}romov's trees, and the {F}arris transform. Discrete
  Appl Math 146(1):51--73. \doi{10.1016/j.dam.2003.01.003}

\bibitem[{Dress et~al.(2014)Dress, Huber, and Steel}]{dress2014associated}
Dress AWM, Huber KT, Steel M (2014) A matroid associated with a phylogenetic
  tree. Discrete Math Theor Comput Sci 16(2):41--55. \doi{10.46298/dmtcs.2078}

\bibitem[{Feichtner and Sturmfels(2005)}]{feichtner2005matroid}
Feichtner EM, Sturmfels B (2005) Matroid polytopes, nested sets and {B}ergman
  fans. Port Math (NS) 62(4):437--468.
  \urlprefix\url{http://eudml.org/doc/52519}

\bibitem[{Gascuel and McKenzie(2004)}]{gascuel2004performance}
Gascuel O, McKenzie A (2004) Performance analysis of hierarchical clustering
  algorithms. J Classification 21(1):3--18. \doi{10.1007/s00357-004-0003-2}

\bibitem[{Gascuel and Steel(2016)}]{gascuel2016stochastic}
Gascuel O, Steel M (2016) A `stochastic safety radius' for distance-based tree
  reconstruction. Algorithmica 74(4):1386--1403.
  \doi{10.1007/s00453-015-0005-y}

\bibitem[{Hayashi(2000)}]{hayashi2000extremum}
Hayashi F (2000) Econometrics. Princeton University Press, Princeton, NJ

\bibitem[{Jaggi et~al.(2008)Jaggi, Katz, and Wagner}]{jaggi2008new}
Jaggi M, Katz G, Wagner U (2008) New results in tropical discrete geometry.
  {P}reprint.
  \urlprefix\url{https://www.m8j.net/math/TropicalDiscreteGeometry.pdf}

\bibitem[{Joswig(2021)}]{ETC}
Joswig M (2021) Essentials of tropical combinatorics, Graduate Studies in
  Mathematics, vol 219. American Mathematical Society, Providence, RI,
  \doi{10.1090/gsm/219}

\bibitem[{Lin and Yoshida(2018)}]{Lin2016TropicalFP}
Lin B, Yoshida R (2018) Tropical {F}ermat-{W}eber points. SIAM J Discrete Math
  32(2):1229--1245. \doi{10.1137/16M1071122}

\bibitem[{Lin et~al.(2017)Lin, Sturmfels, Tang, and Yoshida}]{lin2017convexity}
Lin B, Sturmfels B, Tang X, et~al (2017) Convexity in tree spaces. SIAM J
  Discrete Math 31(3):2015--2038. \doi{10.1137/16M1079841}

\bibitem[{Rinc\'on(2013)}]{rincon2013computing}
Rinc\'on F (2013) Computing tropical linear spaces. J Symbolic Comput
  51:86--98. \doi{10.1016/j.jsc.2012.03.008}

\bibitem[{Sabol et~al.(2024)Sabol, Barnhill, Yoshida, and
  Miura}]{sabol2024polytropes}
Sabol J, Barnhill D, Yoshida R, et~al (2024) Tropical polytropes.
  \urlprefix\url{https://arxiv.org/pdf/2402.14287.pdf}

\bibitem[{Semple and Steel(2003)}]{phylogeneticsTextbook}
Semple C, Steel M (2003) Phylogenetics, Oxford Lecture Series in Mathematics
  and its Applications, vol~24. Oxford University Press, Oxford, Oxford, UK,
  \doi{10.1093/oso/9780198509424.001.0001}

\bibitem[{{The Sage Developers}(2022)}]{sagemath}
{The Sage Developers} (2022) {S}age{M}ath, the {S}age {M}athematics {S}oftware
  {S}ystem. \urlprefix\url{https://www.sagemath.org}

\end{thebibliography}

\end{document}